\documentclass[10pt,british]{article}
\usepackage{amsfonts}
\usepackage{latexsym}
\usepackage{amsmath}
\usepackage{amssymb}
\usepackage{amssymb}
\usepackage{amsthm}
\usepackage{mathrsfs}
\usepackage[pagebackref,colorlinks=true]{hyperref}

\newtheorem{thrm}{Theorem}[section]

\newtheorem{lemma}[thrm]{Lemma}

\newtheorem{prop}[thrm]{Proposition}

\newtheorem{cor}[thrm]{Corollary}

\newtheorem{dfn}[thrm]{Definition}

\newtheorem{rmrk}[thrm]{Remark}

\newtheorem{conv}[thrm]{Convention}

\newtheorem{exam}[thrm]{Example}

\newcommand{\newsection}{    
\setcounter{equation}{0}\section}
\def\appendix#1{\addtocounter{section}{1}\setcounter{equation}{0}
\renewcommand{\thesection}{\Alph{section}}
\section*{Appendix \thesection\protect\indent \parbox[t]{11.15cm}{#1}}
\addcontentsline{toc}{section}{Appendix \thesection\ \ \ #1}}

\newcommand{\be}{\begin{eqnarray}}
\newcommand{\ee}{\end{eqnarray}}
\newcommand{\bea}{\begin{eqnarray}}
\newcommand{\eea}{\end{eqnarray}}
\newcommand{\ba}{\begin{array}}
\newcommand{\ea}{\end{array}}

\newenvironment{frcseries}{\fontfamily{frc}\selectfont}{}

\def\sb {{\nabla}}
\def\LC{{\nabla^g}}
\def\p{{\varphi}}
\def\ph{{\Phi}}
\def\ps{{\Psi}}

\begin{document}
\begin{center}
\vspace*{-1.0cm}


\vspace{1.2 cm} {\Large \bf  Parallel  torsion and $G_2, Spin(7)$ instantons
} 

\vspace{1cm}
 {\large Stefan~ Ivanov${}^{1,2}$,  Alexander~Petkov$^1$ and Luis~Ugarte$^3$}

\vspace{0.5cm}

${}^1$ University of Sofia, Faculty of Mathematics and
Informatics,\\ blvd. James Bourchier 5, 1164, Sofia, Bulgaria

\vspace{0.5cm}
${}^2$ Institute of Mathematics and Informatics,
Bulgarian Academy of Sciences,\\
acad. Georgi Bonchev str. bl.8, 1113 Sofia, Bulgaria

\vspace{0.5cm}
${}^3$ Departamento de Matem\'aticas\,-\,I.U.M.A., 
Universidad de Zaragoza\\
Campus Plaza San Francisco, 
50009 Zaragoza, Spain

\vspace{0.5cm}

\end{center}

\begin{abstract}
Instanton properties of the characteristic connection $\sb$ on an integrable $G_2$ manifold as well as instanton condition of the torsion connection $\sb$ on a $Spin(7)$ manifold  are investigated.  
It is shown that for an integrable $G_2$ manifold with $\sb$-parallel Lee form the curvature of the characteristic connection is a $G_2$ instanton exactly when the torsion 3-form  is $\sb$-parallel. It is observed that  on   a compact $Spin(7)$ manifold with $\sb$ closed  torsion 3-form the torsion connection is a $Spin(7)$ instanton if and only if the torsion 3-form is parallel with respect to the torsion connection.

\medskip

\end{abstract}


\tableofcontents

\setcounter{section}{0}
\setcounter{subsection}{0}



\newsection{Introduction}
Riemannian manifolds with metric connections having totally skew-symmetric torsion and special holonomy received a lot of interest in mathematics and theoretical physics mainly from supersymmetric string theories and supergravity.  The main reason becomes from the Hull-Strominger system which describes the supersymmetric background in heterotic string theories \cite{Str,Hull}. The number of preserved supersymmetries depends on the number  of  parallel spinors with respect to a metric connection $\sb$ with totally skew-symmetric torsion $T$. 
The existence  of a $\nabla$-parallel spinor leads to a restriction of the
holonomy group $Hol(\nabla)$ of the torsion connection $\nabla$.
Namely, $Hol(\nabla)$ has to be contained in $SU(n), Sp(n), G_2, Spin(7)$. 
A detailed analysis of the possible geometries is
carried out in \cite{GMW}. 


In dimension 7 one has to consider a $G_2$ structure. Necessary and sufficient  conditions for a $G_2$ structure  $\p$ to admit a metric connection with torsion 3-form preserving the $G_2$ structure are found in  \cite{FI}, namely the $G_2$ structure has to be integrable,  i.e. $d*\p=\theta\wedge*\p$, where $\theta$ is the Lee form defined below in \eqref{g2li}  and $*$ denotes the Hodge star operator of the Riemannian metric induced by $\p$ (see also \cite{GKMW,FI1,GMW,GMPW,II}). The $G_2$ connection constructed in \cite[Theorem~4.8]{FI} is unique and it is called  \emph{the characteristic connection}. 

From the point of view of physics, 
the compactification of the physical theory leads to the study of models of the form $N^k\times M^{10-k}$, where $N^k$ is a $k$-dimensional Lorentzian manifold and $M^{10-k}$ is a Riemannian spin manifold which encodes the extra dimensions of a supersymmetric vacuum. 
 For application to dimension 7, 
  the integrable $G_2$ structure should be \emph{strictly integrable}, i.e. the scalar product $(d\p,*\p)=0$,  and the Lee form $\theta$ has to be an exact form representing the dilaton, \cite{GKMW}.   It should be mentioned that strictly integrable  $G_2$ structure with an exact Lee form enforce $N=R^{1,2}$
in the compactification. A different compactification ansatz, with $N$ anti-de Sitter space-time, leads to a more general class of solutions with $(d\p,*\p)=\lambda=const.$ \cite{OLS}. The constant  $(d\p,*\p)$ is interpreted as the AdS radius \cite{Oss1,Oss2} ,  \cite[Section~5.2.1]{AMP}. We call this class \emph{integrable $G_2$ structure of constant type} \cite{IS1}.  

In dimension 8,  one has to deal with  a $Spin(7)$ structure. It is shown in \cite[Theorem~1]{I1} that any $Spin(7)$-manifold admits a unique metric connection with totally skew-symmetric torsion  preserving the $Spin(7)$-structure, i.e. there always exists a parallel spinor with respect to the metric connection with torsion 3 form (see also \cite{Fr,Mer} for another proof of this fact). 



 For application to the heterotic string theory in dimension eight, 
 the $Spin(7)$-manifold should be compact 
 and globally conformally balanced which means that the Lee form $\theta$  defined below in \eqref{sg2li} must be an exact form, $\theta=df$ for a smooth function $f$  which represents the dilaton  \cite{GKMW,GMW,GMPW,MSh}.

The Hull-Strominger system in dimension seven, \cite{CGFT,Oss2} (resp. eight) is known as  the $G_2$-Strominger system (resp. the $Spin(7)$-Strominger system).
 It  consists of the supersymmetry  equations and the anomaly cancellation condition. The latter expressed the exterior derivative of the 3-form torsion in terms of a difference of the first Pontrjagin forms of a $G_2$ instanton (resp. $Spin(7)$ instanton) connection on an auxiliary vector bundle and a connection on the tangent bundle. The extra requirements for a solution of the supersymmetry equations  and the anomaly cancellation condition to provide  a supersymmetric vacuum of the theory is given by the $G_2$ instanton (resp. $Spin(7)$ instanton) condition on the connection on the tangent bundle \cite{Iv} (see also \cite{MS,XS}). The $G_2$ instanton (resp. $Spin(7)$ instanton) condition means that the curvature 2-form $R$ of a connection $\sb$ belongs to the Lie algebra $\frak{g}_2$ (resp. Lie algebra $\frak{spin}(7)$ of the Lie group $G_2$ (resp. $Spin(7)$. In general, 
 Hull \cite{Hull}  used the more physically accurate Hull connection   to define
the first Pontrjagin form on the tangent bundle. However, this choice leads to a system of equations, which is not
mathematically closed: e.g. the curvature of the Hull connection is only an  instanton modulo higher order  corrections, see \cite{MS}.

Compact solutions to the $G_2$-Strominger system (resp. to the $Spin(7)$-Strominger system) are constructed in \cite{FIUVdim7-8} with connection on the tangent bundle 
taken as the characteristic connection (resp.  the torsion connection 
of the $Spin(7)$-structure). 
Furthermore, for some of the solutions found on the product $H^5\times T^2$ of the 5-dimensional Heisenberg nilmanifold $H^5$ by the 2-torus, and on  the 7-dimensional generalized Heisenberg nilmanifold $H^7$ (resp. on non-trivial $Spin(7)$ extensions of $H^7$), the 
connection is a $G_2$-instanton (resp. a $Spin(7)$-instanton), thus providing supersymmetric vacua of the theory in dimensions 7 and 8. 

In the case of torsion-free $G_2$-structures, $G_2$-instantons on compact and non-compact manifolds are constructed 
in \cite{Clarke, SW, W} by using different methods, 
and more recently for $G_2$-structures of several non-zero torsion types \cite{BO,LO, Waldron}.

 It is also well known that on a compact $G_2$ manifold  (resp. $Spin(7)$ manifold)  a connection $\sb$ on the tangent bundle with curvature 2-form $R$ is an absolute minimum of the Yang-Mills functional with torsion $YM_T=\int_M||R||^2vol-\int_Mtr(R\wedge R)vol$ if and only if it is a  $G_2$ instanton (resp. $Spin(7)$ instanton).

 The main purpose of the paper is to develop   the $G_2$ instanton  condition of the characteristic connection on a 7-dimensional integrable $G_2$ manifold and the $Spin(7)$ instanton of the torsion connection on an eight dimensional  $Spin(7)$ manifold.  
 
 It is known from  \cite[Lemma~3.4]{I} that the curvature $R$ of a metric connection $\sb$ with  torsion 3-form $T$ is symmetric in exchanging the first and the second pairs,  $R\in S^2\Lambda^2$, 
 if and only if the covariant derivative of the torsion with respect to the torsion connection  is a 4-form, $\sb T\in \Lambda^4$, i.e. the torsion $T$ is a Killing 3-form. If  the holonomy group of a metric connection with torsion 3-form lies in $\frak{g}_2$ (resp. $\frak{spin}(7)$), the condition $\sb T\in \Lambda^4$ implies that the curvature is a $G_2$ instanton (resp.  $Spin(7)$ instanton). In particular, if the torsion is parallel with respect to this connection then its  curvature is an instanton.
 
The main object of interest in the paper is to investigate when the converse statement holds, namely, when the $G_2$ or $Spin(7)$ instanton condition implies the torsion is parallel.  We note that $G_2$ and $Spin(7)$ instanton manifolds with Killing torsion form   are investigated in \cite{IS1} and \cite{Iap}, respectively. In this note we dropped the Killing torsion condition. 
 
In the $G_2$ case, we show the following

\begin{thrm}\label{instm}
Let $(M,\p)$ be  an integrable $G_2$ manifold with $\sb$-parallel Lee form and the  curvature of the characteristic  connection $\sb$   is a $G_2$-instanton, i.e.
\[d*\p=\theta\wedge*\p, \qquad \sb\theta=0, \quad R\in \frak{g}_2\otimes\frak{g}_2.\]
Then the torsion 3-form is parallel with respect to the characteristic connection, $\sb T=0$.

In particular, the $G_2$ manifold is of constant type, the characteristic Ricci tensor is symmetric, $\sb$-parallel  and $\sb dT=0$.
\end{thrm}
The main observation in the proof of the theorem is Proposition~\ref{instm1} which says that under the conditions of the theorem the four form $$d^{\sb}T=4Alt(\sb T)=0,$$ where $Alt(\sb T)$ stand for the alternation of $\sb T$ (see \eqref{dh} below). 

In terms of $dT$ and the four form $\sigma^T$ introduced below in \eqref{sigma}, the condition $d^{\sb}T=0$ is equivalent to $dT=2\sigma^T$ (see \eqref{inst6} below). Note, that if $\sb T=0$ then automatically $d^{\sb} T=0$ and $dT=2\sigma^T$.


With the help of Theorem~\ref{instm} we improve \cite[Theorem~1.1]{IS1} as follows
\begin{thrm}\label{instqjm}
Let $(M,\p)$ be  an integrable $G_2$ manifold of constant type and the  curvature of the characteristic  connection $\sb$   is a Ricci flat  $G_2$-instanton, i.e.
\[d*\p=\theta\wedge*\p \qquad (d\p,*\p)=const.,\quad R\in \frak{g}_2\otimes\frak{g}_2, \quad Ric=0.\]
Then the torsion 3-form is parallel with respect to the Levi-Civita connection and to the characteristic connection, $\LC T=\sb T=0.$ 

In particular, the torsion 3-form  is harmonic, $\delta T=dT=0$ and the  4-form $\sigma^T=0$. 
\end{thrm}

Since on a co-calibrated $G_2$ manifold the Lee form vanishes, $\theta=0$, Theorem~\ref{instm} implies
\begin{cor}\label{instco}
Let $(M,\p)$ be  a  co-calibrated $G_2$ manifold  and the  curvature of the characteristic  connection $\sb$   is a $G_2$-instanton, i.e.
\[d*\p=0, \quad R\in \frak{g}_2\otimes\frak{g}_2.\]
Then the torsion 3-form is parallel with respect to the characteristic connection, $\sb T=0$.
\end{cor}

Integrable $G_2$ structures  with parallel torsion 3-form with respect to the characteristic connection are investigated in  \cite{F,AFer,CMS}   and a large number of examples are given there.  
In the case of left-invariant $G_2$-structures on Lie groups, a classification of $2$-step nilpotent Lie groups and co-calibrated $G_2$-structures on them for which the characteristic connection 
satisfies the $G_2$-instanton condition 
is obtained in \cite{CdBM}. From this classification it follows that the $G_2$-instantons given in \cite{FIUVdim7-8} 
are the only ones of purely co-calibrated type (i.e. $(d\p,*\p)=0=d*\p$) in the class of $2$-step nilpotent Lie groups. 
It is also proved in \cite[Theorem 1.2]{CdBM} that for left-invariant co-calibrated $2$-step nilpotent Lie groups, the $G_2$-instanton condition implies $\nabla T=0$, so our Corollary~\ref{instco} 
provides an extension of this result to any co-calibrated $G_2$ manifold. 

Integrable $G_2$ structures  with parallel torsion 3-form with respect to the characteristic connection are recently classified,  up to naturally reductive homogeneous spaces and nearly parallel $G_2$ structures, in\cite{MSem}.

Note that integrable $G_2$ structures  with $\sb$-parallel torsion 3-form  have co-closed Lee form. More general,   due to  \cite[Theorem~3.1]{FI1},  for any  integrable $G_2$ structure  on a compact manifold there exists an unique integrable $G_2$ structure conformal to the original one with co-closed Lee form,  called \emph{the Gauduchon $G_2$ structure}.

In the compact case we prove 
\begin{thrm}\label{cmpga}
Let $(M,\p)$ be a compact integrable $G_2$ manifold of constant type with a Gauduchon $G_2$ structure, $\delta\theta=0$.  

The characteristic connection is a $G_2$-instanton if and only if the torsion 3-form is parallel with respect to the characteristic connection, $\sb T=0$. 
\end{thrm}

We also observe in Theorem~\ref{g2clost} that on a compact integrable $G_2$ manifold with closed torsion, $dT=0$, the characteristic connection is a $G_2$ instanton if and only if the torsion is parallel with respect to the characteristic connection.

For $Spin(7)$ manifold we  show the following

\begin{thrm}\label{main2}
Let $(M,\p)$ be  a compact  $Spin(7)$ manifold. 

The  curvature of the torsion  connection $\sb$   is a $Spin(7)$-instanton and $d^{\sb}T=0$, i.e.
\[R\in \frak{spin}(7)\otimes\frak{spin}(7), \qquad dT=2\sigma^T , \] 
if and only if the torsion 3-form is parallel with respect to the torsion connection, $\sb T=0$.

In particular, the  Ricci tensor of the torsion connection is symmetric, $\sb$-parallel  and $\sb dT=0$.
\end{thrm}

We show in Theorem~\ref{sp7clost} that on a compact $Spin(7)$ manifold with closed torsion the torsion connection is a $Spin(7)$ instanton if and only if the torsion is parallel with respect to the characteristic connection.

In the non-compact case we observe 
\begin{thrm}\label{mainsp0}
Let $(M,\ps)$ be  a $Spin(7)$ manifold.

If the Lee form is closed,  the torsion  connection $\sb$   is a $Spin(7)$-instanton and $d^{\sb}T=0$, 
\[d\theta=0,\qquad R\in \frak{spin}(7)\otimes\frak{spin}(7), \qquad dT=2\sigma^T,\]
then the torsion 3-form is parallel with respect to the torsion connection, $\sb T=0$.

In this case the  Ricci tensor of the torsion connection is symmetric, $\sb$-parallel  and $\sb dT=0$.
\end{thrm}

\begin{rmrk}
We remark that the converse in Theorem~\ref{mainsp0} is not true. We construct in Example~\ref{exdtit} a $Spin(7)$ manifold having parallel torsion with respect to the torsion connection with non-closed Lee form.
\end{rmrk}

In the general non-compact case we have
\begin{thrm}\label{mainsp}
Let $(M,\ps)$ be  a $Spin(7)$ manifold. 

The  curvature of the torsion  connection $\sb$   is a $Spin(7)$-instanton and $d^{\sb}T=\delta T=0$, i.e.
\[R\in \frak{spin}(7)\otimes\frak{spin}(7), \qquad dT=2\sigma^T,\qquad \delta T=0 , \]
if and only if the torsion 3-form is parallel with respect to the characteristic connection, $\sb T=0$.
\end{thrm}
On a balanced $Spin(7)$ manifold the Lee form vanishes and we derive
\begin{cor}\label{mainspc}
Let $(M,\ps)$ be  a balanced $Spin(7)$ manifold. 

The  curvature of the torsion  connection $\sb$   is a $Spin(7)$-instanton and $d^{\sb}T=0$, i.e.
\[R\in \frak{spin}(7)\otimes\frak{spin}(7), \qquad dT=2\sigma^T , \]
if and only if the torsion 3-form is parallel with respect to the characteristic connection, $\sb T=0$.
\end{cor}

Note that a $Spin(7)$ structures  with $\sb$-parallel torsion 3-form  have co-closed Lee form. More general,   due to  \cite[Theorem~4.3]{I1},  for any  $Spin(7)$ structure  on a compact manifold there exists an unique $Spin(7)$ structure in the same conformal with co-closed Lee form,  called \emph{the Gauduchon $Spin(7)$ structure}.

\begin{thrm}\label{spga}
Let $(M,\tilde{\ps})$ be a compact $Spin(7)$ manifold with closed Lee form, $d\tilde{\theta}=0$.

If the torsion connection  $\nabla$ of the  Gauduchon $Spin(7)$ structure $\ps=e^f\tilde{\ps}$  is a $Spin(7)$-instanton then its  Lee form $\theta$ is parallel with respect to the torsion connection, $\sb\theta=0$, and the 4-form $d^{\sb}T\in\Lambda^4_{27}$.

In particular, the 4-form $d^{\sb}T$ is self-dual, $*d^{\sb}T=d^{\sb}T$.
\end{thrm}

We remark that on compact $G_2$ and $Spin(7)$ manifolds with closed torsion there exists a canonical symmetry reduction discovered recently by A. Kennon and J. Streets in \cite{KenStr}.
On the other hand, a rigidity theorem for Riemannian manifolds
that admit a connection     with torsion a 3-form, which is both closed and    covariantly
constant, is proved by G. Papadoulos in \cite{Pap}. 

\begin{conv}
Everywhere in the paper we will make no difference between tensors and the corresponding forms via the metric as well as we will  use Einstein summation conventions, i.e. repeated Latin  indices are summed over. The $*$ denotes the Hodge star operator of the Riemannian metric induced by $G_2$ structure $\p$ or by the Spin(7) structure $\ps$.

\end{conv}

\section{Preliminaries}
In this section, we recall some known curvature properties of a metric connection with totally skew-symmetric torsion on a Riemannian manifold. 

On a Riemannian manifold $(M,g)$ of dimension $n$ any metric connection $\sb$ with totally skew-symmetric torsion $T$ is connected with the Levi-Civita connection $\sb^g$ of the metric $g$ by
\begin{equation}\label{tsym}
\sb^g=\sb- \frac12T \quad \textnormal{leading to} \quad \LC T=\sb T+\frac12\sigma^T,
\end{equation}
where the 4-form $\sigma^T$, introduced in \cite{FI}, is defined by
 \begin{equation}\label{sigma}
 \sigma ^T(X,Y,Z,V)=\frac12\sum_{j=1}^n(e_j\lrcorner T)\wedge(e_j\lrcorner T)(X,Y,Z,V), 
\end{equation} 
   $(e_a\lrcorner T)(X,Y)=T(e_a,X,Y)$ is the interior multiplication and $\{e_1,\dots,e_n\}$ is an orthonormal  basis.

The properties of the 4-form $\sigma^T$ are studied in detail in \cite{AFF} where it is shown that $\sigma^T$ measures the `degeneracy' of the 3-form $T$.

The exterior derivative $dT$ has the following  expression (see e.g. \cite{I,IP2,FI})
\begin{equation}\label{dh}
\begin{split}
dT(X,Y,Z,V)=d^{\sb}T(X,Y,Z,V) +2\sigma^T(X,Y,Z,V), \quad \textnormal{where}\\
d^{\sb}T(X,Y,Z,V)=(\nabla_XT)(Y,Z,V)+(\nabla_YT)(Z,X,V)+(\nabla_ZT)(X,Y,V)-(\nabla_VT)(X,Y,Z).
 \end{split}
 \end{equation}
 For the curvature of $\sb$ we use the convention $ R(X,Y)Z=[\nabla_X,\nabla_Y]Z -
 \nabla_{[X,Y]}Z$ and $ R(X,Y,Z,V)=g(R(X,Y)Z,V)$. It has the well known properties
$
 R(X,Y,Z,V)=-R(Y,X,Z,V)=-R(X,Y,V,Z).
 $
 
  The first Bianchi identity for $\nabla$ can be written in the form (see e.g. \cite{I,IP2,FI})
 \begin{equation}\label{1bi}
 \begin{split}
 R(X,Y,Z,V)+ R(Y,Z,X,V)+ R(Z,X,Y,V)\\ 
 =dT(X,Y,Z,V)-\sigma^T(X,Y,Z,V)+(\nabla_VT)(X,Y,Z).
 \end{split}
 \end{equation}
It is proved in \cite[p. 307]{FI} that the curvature of  a metric connection $\sb$ with totally skew-symmetric torsion $T$  satisfies also the  identity
 \begin{equation}\label{gen}
 \begin{split}
 R(X,Y,Z,V)+ R(Y,Z,X,V)+ R(Z,X,Y,V)-R(V,X,Y,Z)-R(V,Y,Z,X)-R(V,Z,X,Y)\\
 =\frac32dT(X,Y,Z,V)-\sigma^T(X,Y,Z,V).
 \end{split}
 \end{equation}
 One gets from \eqref{gen} and \eqref{1bi} that the curvature of the torsion connection satisfies the identity  
 \cite[Proposition~2.1]{IS}
 \begin{equation}\label{1bi1}
 \begin{split}
R(V,X,Y,Z)+R(V,Y,Z,X)+R(V,Z,X,Y)= -\frac12dT(X,Y,Z,V)+(\nabla_VT)(X,Y,Z). 
 \end{split}
 \end{equation}
 
It is known from  \cite[Lemma~3.4]{I} that a metric connection $\sb$ with  torsion 3-form $T$ has curvature $R\in S^2\Lambda^2$, i.e. it satisfies 
 \begin{equation}\label{r4}
 R(X,Y,Z,V)=R(Z,V,X,Y) ,
 \end{equation}
 if and only if the covariant derivative of the torsion with respect to the torsion connection  is a 4-form, i.e. the torsion $T$ is a Killing 3-form. Killing forms on Riemannian manifold,  introduced by Yano in \cite{Yan}, are  natural generalization of Killing vector fields and have a lot of interest  in mathematics and physics (see e.g. \cite{USem} and references therein). 
 
\begin{lemma}\cite[Lemma~3.4]{I} The next equivalences hold for a metric connection with torsion 3-form
\begin{equation}\label{4form}
(\sb_XT)(Y,Z,V)=-(\sb_YT)(X,Z,V) \Longleftrightarrow R(X,Y,Z,V)=R(Z,V,X,Y) ) \Longleftrightarrow dT=4\LC T.
\end{equation}
\end{lemma}

\noindent The   Ricci tensors and scalar curvatures of  $\LC$ and $\sb$ are related by (\cite[Section~2]{FI}, \cite [Prop. 3.18]{GFS})
\begin{equation}\label{rics}
\begin{split}
Ric^g(X,Y)=Ric(X,Y)+\frac12 (\delta T)(X,Y)+\frac14\sum_{i=1}^ng\Big(T(X,e_i),T(Y,e_i)\Big) , \\
Scal^g=Scal+\frac14||T||^2,\qquad Ric(X,Y)-Ric(Y,X)=-(\delta T)(X,Y),
\end{split}
\end{equation}
where $\delta=(-1)^{np+n+1}*d*$ is the co-differential acting on $p$-forms and $*$ is the Hodge star operator satisfying $*^2=(-1)^{p(n-p)}$.


One  has the general identities for $\alpha\in\Lambda^1$ and $\beta\in \Lambda^k$ 
\begin{equation}\label{1star}
\begin{split}
*(\alpha\lrcorner\beta)=(-1)^{k+1}(\alpha\wedge*\beta);\qquad (\alpha\lrcorner\beta)=(-1)^{n(k+1)}*(\alpha\wedge*\beta);\\
*(\alpha\lrcorner*\beta)=(-1)^{n(k+1)+1}(\alpha\wedge\beta);\qquad (\alpha\lrcorner*\beta)=(-1)^{k}*(\alpha\wedge\beta).
\end{split}
\end{equation}
Denote by $\delta^{\sb}T$ the negative  trace of $\sb T$, $\delta^{\sb}T(X,Y)=-(\sb_{e_i}T)(e_i,X,Y)$.
 
 It follows  from \eqref{dh} and \eqref{tsym} that 
 \begin{equation}\label{inst6}
 d^{\sb}T=0 \Longleftrightarrow dT=2\sigma^T;\qquad \delta^{\sb}T=\delta T.
 \end{equation}

\section{$G_2$ structure}

We recall some notions of $G_2$ geometry. Endow $\mathbb R^7$ with its standard orientation and
inner product. Let $\{e_1,\dots,e_7\}$ be an oriented orthonormal basis which we identify with
the dual basis via the inner product. Write $e_{i_1 i_2\dots i_p}$ for the
monomial $e_{i_1} \wedge e_{i_2} \wedge \dots \wedge e_{i_p}$
and consider the three-form $\p$ on $\mathbb R^7$ given by
\begin{equation}\label{g2}
\begin{split}
\p=e_{127}+e_{135}-e_{146}-e_{236}-e_{245}+e_{347}+e_{567}.
\end{split}
\end{equation}
The subgroup of $GL(7)$ fixing $\p$ is the exceptional Lie group $G_2$. It is a compact,
connected, simply-connected, simple Lie subgroup of $SO(7)$ of dimension 14 \cite{Br}.  The Lie algebra is denoted by
 $\frak{g}_2$ and it is isomorphic to the 2-forms satisfying 7 linear equations, namely 
 $\frak{g}_2\cong  \{\alpha\in \Lambda^2(M) \vert
  *(\alpha\wedge\p) =- \alpha\}.$ 
  
The 3-form
$\p$ corresponds to a real spinor and therefore,
$G_2$ can be identified as the isotropy group of a non-trivial
real spinor.

The Hodge star operator supplies the 4-form $\ph=*\p$ given by
\begin{equation*}
\begin{split}
\ph=*\p=e_{1234}+e_{3456}+e_{1256}-e_{2467}+e_{1367}+e_{2357}+e_{1457}.
\end{split}
\end{equation*}
We recall that in dimension seven, the Hodge star operator satisfies $*^2=1$  and has the properties
\begin{equation}\label{star}
\begin{split}
*(\beta\wedge\p)=\beta\lrcorner *\p, \quad \beta\in \Lambda^2, \qquad 
*(\beta\wedge *\p)=\beta\lrcorner \p, \quad \beta\in \Lambda^2.
\end{split}
\end{equation}

We let the expressions 
$$
 \p=\frac16\p_{ijk}e_{ijk}, \quad  \ph=\frac1{24}\ph_{ijkl}e_{ijkl}
$$
and have the  identities (c.f. \cite{Br1,Kar1,Kar})
\begin{equation}\label{iden}
\begin{split}
\p_{ijk}\p_{ajk}=6\delta_{ia};\qquad \p_{ijk}\p_{ijk}=42;\\
 \p_{ijk}\p_{abk}=\delta_{ia}\delta_{jb}-\delta_{ib}\delta_{ja}+\ph_{ijab};\quad
\p_{ijk}\ph_{abjk}=4\p_{iab};\\
\p_{ijk}\ph_{kabc}=\delta_{ia}\p_{jbc}+\delta_{ib}\p_{ajc}+\delta_{ic}\p_{abj}-\delta_{aj}\p_{ibc}-\delta_{bj}\p_{aic}-\delta_{cj}\p_{abi}.\\
\end{split}
\end{equation}
A $G_2$ structure on a 7-manifold $M$ is a reduction of the structure group
of the tangent bundle to the exceptional Lie group $G_2$. Equivalently, there exists a nowhere
vanishing differential three-form $\p$ on $M$ and local frames of the cotangent bundle with
respect to which $\p$ takes the form \eqref{g2}. The three-form $\p$ is called the fundamental
form of the $G_2$ manifold $M$ \cite{Bo}.
We will say that the pair $(M, \p)$ is a $G_2$ manifold with $G_2$ structure (determined
by) $\p$.  Alternatively, a $G_2$ structure can be described by the existence of a two-fold
vector cross product  on the tangent spaces of $M$ (see e.g. \cite{Gr}).

It is well known that the fundamental form of a $G_2$ manifold determines a Riemannian metric which  is referred to as the metric induced by $\p$. We write $\LC$ for the associated Levi-Civita
connection. 


The action of $G_2$  on the tangent space induces an
action of $G_2$ on $\Lambda^k(M)$ splitting the exterior algebra into orthogonal subspaces, where
$\Lambda^k_l$ corresponds to an $l$-dimensional $G_2$-irreducible subspace of $\Lambda^k$:
\begin{equation*}
\begin{split}
\Lambda^1(M)=\Lambda^1_7,  \quad \Lambda^2(M)=\Lambda^2_7\oplus\Lambda^2_{14}, \qquad \Lambda^3(M)=\Lambda^3_1\oplus\Lambda^3_7\oplus\Lambda^3_{27},
\end{split}
\end{equation*}
where
\begin{equation}\label{dec2}
\begin{split}
\Lambda^2_7=\{\phi\in \Lambda^2(M) | *(\phi\wedge\p)=2\phi\};\\
\Lambda^2_{14}=\{\phi\in \Lambda^2(M) | *(\phi\wedge\p)=-\phi\}\cong \frak{g}_2;\\
\Lambda^3_1=t\p,\quad t\in \mathbb R;\\
\Lambda^3_7=\{*(\alpha\wedge\p) |  \alpha\in\Lambda^1\}=\{\alpha\lrcorner\ph\};\\
\Lambda^3_{27}=\{\gamma\in \Lambda^3(M) | \gamma\wedge\p=\gamma\wedge\ph=0\}.
\end{split}
\end{equation}
We recall  the next algebraic fact stated  in the proof of \cite[Theorem~5.4]{FI} (see a proof of it in \cite{IS1}) .
\begin{prop}\cite[p. 319]{FI}\label{4-for}
Let $A$ be a 4-form and define the 3-forms $B_X=(X\lrcorner A)$ for any $X\in T_pM$. If  the 3-forms $B_X\in\Lambda^3_{27}$ then the 4-form $A$ vanishes identically, $A=0$
\end{prop}

\begin{rmrk}There is another different orientation convention for $G_2$ structures.
In the other convention, the eigenvalues of the operator $\beta\rightarrow *(\beta\wedge\p)$ are -2 and +1 instead of +2 and -1, respectively. 
\end{rmrk}

In~\cite{FG}, Fernandez and Gray divide $G_2$ manifolds into 16 classes
according to how the covariant derivative $\LC\p$ 
behaves with respect to its decomposition into $G_2$ irreducible components
(see also~\cite{CSal,GKMW,Br1}).  If the fundamental form is parallel with respect to
the Levi-Civita connection, $\LC\p=0$,
 then the Riemannian holonomy group is contained
in $G_2$. In this case the induced metric on the
$G_2$ manifold is Ricci-flat, a fact first observed by Bonan~\cite{Bo}.  It
was also shown in \cite{FG} that a $G_2$ manifold
is parallel precisely when the fundamental form is harmonic, i.e. $d\p=d*\p=0$.
The first examples of complete parallel $G_2$ manifolds were constructed by Bryant and
Salamon~\cite{BS,Gibb}.  Compact examples of parallel $G_2$ manifolds were
obtained first by Joyce~\cite{J1,J2,J3} and  with another construction by Kovalev~\cite{Kov}.

The Lee form $\theta$ is defined by \cite{Cabr} (see also \cite{Br})
\begin{equation}\label{g2li}
\theta=-\frac{1}{3}*(* d\p\wedge\p)=\frac13*(*d*\p\wedge *\p)=-\frac13*(\delta\p\wedge*\p)=-\frac13\delta\p\lrcorner\p,
\end{equation}
where $\delta=(-1)^k*d*$ is the codifferential  acting on $k$-forms  and one applies \eqref{star} to get the last identity.

The failure of the holonomy group of the Levi-Civita connection $\sb^g$ of the metric $g$ to reduce to $G_2$
can also be measured by the intrinsic torsion $\tau$, which is identified with $d\p$ an $d*\p=d\ph$, 
and can be decomposed into four basic classes \cite{CSal,Br1}, $\tau \in W_1\oplus W_7 \oplus W_{14}\oplus W_{27}$ which gives another description of the Fern\'andez-Gray classification \cite{FG}.  We list below those of them which
we will use later.

\noindent - $\tau \in W_1$. The class of nearly parallel  (weak holonomy) $G_2$ manifold defined by
$d\p=const.*\p, \quad d*\p=0$.

\noindent - $\tau \in W_7$. The class of locally conformally parallel $G_2$ spaces characterized by
$d*\p=\theta\wedge *\p, \quad d\p=\frac34\theta\wedge\p$.
 - $\tau \in W_{27}$. The class of pure integrable  $G_2$ manifolds determined by  $d\p\wedge\p=0$ and $d*\p=0$.

\noindent - $\tau \in W_1\oplus W_{27}$. The class of cocalibrated $G_2$ manifold, 
determined by the condition $d*\p=0$. 
 
\noindent - $\tau \in W_1\oplus W_7\oplus W_{27}$. The class of integrable $G_2$ manifold determined by the condition $d*\p=\theta\wedge *\p$.  An analog of the Dolbeault cohomology is investigated in \cite{FUg}.  In this class, the exterior derivative of the Lee form lies in the Lee algebra $\frak{g}_2$, $d\theta\in \Lambda^2_{14}\cong\frak{g}_2$ \cite{Kar1}. 
This is the  class which we are interested in.

 \noindent- $\tau \in W_7\oplus W_{27}$. This class is determined by the conditions $d\p\wedge\p=0$ and $d*\p=\theta\wedge*\p$ and is of great interest in supersymmetric heterotic string theories in dimension seven  \cite{GKMW,FI,FI1,GMW,GMPW,OLS}. We call this class \emph{strictly}  integrable   $G_2$ manifolds .

An important sub-class of the integrable $G_2$ manifolds is determined in the next
\begin{dfn}\label{ctype}
An integrable $G_2$ structure is said to be of constant type if the function $(d\p,*\p)=const$.
\end{dfn}
For example, the nearly parallel as well as the strictly integrable $G_2$ manifolds are integrable of constant type. The integrable $G_2$ manifolds of constant type appear also in  the $G_2$ heterotic supergravity where  the constant  $(d\p,*\p)$ is interpreted as the AdS radius \cite{Oss1,Oss2} see also  \cite[Section~5.2.1]{AMP}.

If the Lee form of an integrable $G_2$ structure  vanishes, $\theta=0$ then the $G_2$ structure is
co-calibrated. If the Lee form of an integrable $G_2$ structure is closed,
$d\theta=0$ then the $G_2$ structure is locally conformally
equivalent to a co-calibrated one \cite{FI1} (see also \cite{Kar1}) and if the Lee form is an exact form then it is (globally) conformal to a co-calibrated one.  
It is known due to  \cite[Theorem~3.1]{FI1} that for any  integrable $G_2$ structure  on a compact manifold, there exists a unique integrable $G_2$ structure conformal to the original one with co-closed Lee form,  called \emph{the Gauduchon $G_2$ structure}.

We recall the following
\begin{dfn}
The curvature $R$ of a linear connection on a $G_2$ manifold  is a $G_2$-instanton if  the curvature 2-form lies in the Lie algebra $\frak{g}_2\cong\Lambda^2_{14}$. This is equivalent to  the  identities:
\begin{equation}\label{rri}
R_{abij}\p_{abk}=0 \Longleftrightarrow R_{abij}\ph_{abkl}=-2R_{klij}.
\end{equation}
\end{dfn}
\section{The $G_2$-connection with skew-symmetric torsion}
The necessary and sufficient  conditions   a 7-dimensional manifold with a $G_2$ structure to admit a metric connection with torsion 3-form preserving the $G_2$ structure are found in \cite{FI} ( see also \cite{GKMW,FI1,GMW,GMPW}). 
\begin{thrm}\cite[Theorem~4.8]{FI}\label{cythm1}
Let $(M,\p)$ be a  smooth  manifold with a $G_2$ structure $\p$.

The next two conditions are equivalent
\begin{enumerate}
\item[a)] The $G_2$ structure $\p$ is integrable,
\begin{equation}\label{cycon}
\begin{split}
d*\p=\theta\wedge *\p.
\end{split}
\end{equation}
\item[b)] There exists a unique $G_2$-connection $\sb$ with torsion 3-form preserving the $G_2$ structure,

$
\sb g=\sb\p=\sb\ph=0.
$ 
The torsion of $\sb$ is given by
\begin{equation}\label{torcy}
\begin{split}
T=-*d\p +*(\theta\wedge\p)+\frac{1}{6}(d\p,*\p)\p.
\end{split}
\end{equation}
\end{enumerate}
\end{thrm}
The unique linear connection $\sb$  preserving the $G_2$ structure with totally skew-symmetric torsion is called \emph{the characteristic connection}.  The curvature and the Ricci tensor of $\sb$  will be called \emph{characteristic curvature} and \emph{characteristic Ricci tensor}, respectively.

If the $G_2$ structure is nearly parallel then  the torsion is parallel with respect to the characteristic connection, $\sb T=0$  \cite{FI}.


\subsection{The torsion and the Ricci tensor of the characteristic connection}
We obtain from \eqref{torcy} using \eqref{star} that
\begin{equation}\label{torcy1}
\begin{split}
T=-*d\p +*(\theta\wedge\p)+\frac{1}{6}(d\p,\ph)\p 
=-*d*\Phi-\theta\lrcorner\Phi+\frac{1}{6}(d\p,\ph)\p
=-\delta\Phi-\theta\lrcorner\Phi+\frac{1}{6}(d\p,\ph)\p.
\end{split}
\end{equation}
Write $\delta\ph$  in terms $\LC$  and then in terms of $\sb$ using \eqref{tsym} and  $\sb\Phi=0$ to get 
\begin{equation}\label{dphi}
\begin{split}
-\delta\Phi_{klm}
=-\frac12T_{jsk}\Phi_{jslm}+\frac12T_{jsl}\Phi_{jskm}-\frac12T_{jsm}\Phi_{jskl}.
\end{split}
\end{equation}
Substituting \eqref{dphi} into \eqref{torcy1}, we obtain  the following formula of the 3-form torsion $T$ from \cite{IS1},
\begin{equation}\label{torcy2}
T_{klm}=-\frac12T_{jsk}\Phi_{jslm}+\frac12T_{jsl}\Phi_{jskm}-\frac12T_{jsm}\Phi_{jskl}-\theta_s\Phi_{sklm}+\lambda\p_{klm}, 
\end{equation}
where the function $\lambda$ is defined by the  scalar product
\begin{equation}\label{lm}
\lambda =\frac16(d\p,\ph)=\frac1{42}d\p_{ijkl}\ph_{ijkl}=\frac1{36}\delta\ph_{klm}\p_{klm}.
\end{equation}
Applying \eqref{iden},  it is easy to check from \eqref{g2li} and \eqref{torcy2} that  $\theta$ and $\lambda$ can be written in terms of $T$ 
\begin{equation}\label{tit}
\begin{split}
\theta_i=\frac16T_{jkl}\Phi_{jkli}, \qquad \lambda=\frac16T_{klm}\p_{klm}.
\end{split}
\end{equation}
Similarly, we obtain the next identities
\begin{equation}\label{lmt0}
\begin{split}
T_{kli}\p_{klj}-T_{klj}\p_{kli}=-2\theta_s\p_{sij}.\\
\sigma^T_{iabc}\p_{abc}=-3T_{abs}\p_{abc}T_{sci}=3\theta_s\p_{skt}T_{kti}.
\end{split}
\end{equation}
Denote by $d^{\sb}\theta$ the skew-symmetric part of $\sb\theta$,  $d^{\sb}\theta(X,Y)=(\sb_X\theta)Y-(\sb_Y\theta)X$. 

The next result is observed in \cite[(4.16)]{IS1} and we include here the proof for completeness. 

\begin{prop}\cite[(4.16)]{IS1}\label{pp1}
On an integrable     $G_2$ manifold $(M,\p)$ the co-differential of the torsion is given by
\begin{equation}\label{deltaT}
\delta T=d^{\sb}\theta-d\lambda\lrcorner\p.
\end{equation}
\end{prop}
\begin{proof}
We calculate from \eqref{torcy} using \eqref{star}, \eqref{cycon},  \eqref{dec2} and the fact obseved in \cite{Kar1} that $d\theta\in\Lambda^2_{14}$ 
\begin{multline}\label{rrr}
-\delta T=*d*T=*(d\theta\wedge\p)-*(\theta\wedge d\p)+*(d\lambda\wedge\ph) +*(\lambda\theta\wedge\ph)\\=-d\theta-*(\theta\wedge d\p)+*[d\lambda+\lambda\theta)\wedge\ph]=-d\theta-\theta\lrcorner*d\p+(d\lambda+\lambda\theta)\lrcorner\p\\=-d\theta-\theta\lrcorner\delta\ph+(d\lambda+\lambda\theta)\lrcorner\p=-d\theta+\theta\lrcorner T-\lambda\theta\lrcorner\p+(d\lambda+\lambda\theta)\lrcorner\p=-d\theta+\theta\lrcorner T+d\lambda\lrcorner\p, 
\end{multline}
where we have applied \eqref{torcy1} in the third line.

On the other hand, \eqref{tsym} yields
\begin{equation}\label{ttg2}
d\theta=d^{\sb}\theta+\theta\lrcorner T,
\end{equation}
 which substituted into \eqref{rrr} gives \eqref{deltaT}.
 \end{proof}

We obtain from Proposition~\ref{pp1} and \eqref{rics} that 
on an integrable     $G_2$ manifold $(M,\p)$ the characteristic Ricci tensor  is symmetric, $Ric(X,Y)=Ric(Y,X)$ if and only if the two form $d^{\sb}\theta$  is given by
 \begin{equation}\label{newsymth}
d^{\sb}\theta=d\lambda\lrcorner\p \in\Lambda^2_7. 
 \end{equation}



Explicit formulas of the characteristic Ricci tensor of an integrable $G_2$ manifold are presented in \cite{FI,FI1}. Below, we give the proof from \cite{IS1} for completeness. We have
\begin{thrm}\cite{FI, FI1}\label{thRic}
The characteristic Ricci tensor $Ric$ and its scalar curvature $Scal$  are given by
\begin{equation}\label{ricg2}
\begin{split}
Ric_{ij}=\frac1{12}dT_{iabc}\ph_{jabc}-\sb_i\theta_j, \quad
Scal=3\delta\theta+2||\theta||^2-\frac13||T||^2+2\lambda^2.
\end{split}
\end{equation}
The next  identities hold
\begin{equation}\label{dtt}
\begin{split}
dT_{iabc}\p_{abc}+2\sb_iT_{abc}\p_{abc}=
dT_{iabc}\p_{abc}+12d\lambda_i=0;\\
 3\sb_aT_{bci}\p_{abc}=2\sigma^T_{iabc}\p_{abc}+18d\lambda_i=6\theta_sT_{skt}\p_{kti}+18d\lambda_i.
\end{split}
\end{equation}
\end{thrm}
\begin{proof}
 Since $\sb\p=0$ the holonomy group of the characteristic connection lies in the Lie algebra $g_2$, i.e. 
\begin{equation}\label{rr}
\begin{split}
R_{ijab}\p_{abk}=0 \Longleftrightarrow R_{ijab}\ph_{abkl}=-2R_{ijkl}.
\end{split}
\end{equation}
We have from \eqref{rr} using \eqref{1bi1}, \eqref{tit} and \eqref{dh}  that the Ricci tensor $Ric$ of  $\sb$ is given by
\begin{multline}\label{ricdt}
2Ric_{ij}=R_{iabc}\ph_{jabc}=\frac13\big[R_{iabc}+R_{ibca}+R_{icab} \big]\ph_{jabc}=\frac16dT_{iabc}\ph_{jabc}+\frac13\sb_iT_{abc}\ph_{jabc}.
\end{multline}
Apply \eqref{tit} to complete the proof of the first identity in \eqref{ricg2}.
Similarly, we have
\[0=R_{iabc}\p_{abc}=\frac13\big[R_{iabc}+R_{ibca}+R_{icab} \big]\p_{abc}=\frac16dT_{iabc}\p_{abc}+\frac13\sb_iT_{abc}\p_{abc}\]
which proves the first equality in \eqref{dtt}. Apply \eqref{dh} to achieve the second and \eqref{lmt0} to get the third.

We obtain  from \eqref{torcy2} using \eqref{iden} 
\begin{equation}\label{ng4}
\begin{split}
\sigma^T_{jabc}\ph_{jabc}=3T_{jas}T_{bcs}\ph_{jabc}=-2||T||^2+12||\theta||^2+12\lambda^2
\end{split}
\end{equation}
We calculate from \eqref{dh} applying \eqref{tit}, \eqref{ng4} 
\begin{equation}\label{g221}
\begin{split}
dT_{jabc}\ph_{jabc}=4\sb_jT_{abc}\ph_{jabc}+2\sigma^T_{jabc}\ph_{jabc}=-24\sb_j\theta_j-4||T||^2+24||\theta||^2+24\lambda^2.
\end{split}
\end{equation}
Take the trace in the first identity in \eqref{ricg2} substitute  \eqref{g221} into the obtained equality and use \eqref{lm} 
to get the second identity in \eqref{ricg2}. 
\end{proof}
\begin{rmrk}\label{remg21}
It follows from \eqref{ricg2}, \eqref{dh}, \eqref{sigma} and \eqref{inst6} that if  $\sb T=0$ then $\delta T=  \sb Ric=\sb dT=0$ and $d(Scal)=0$.
\end{rmrk}
\begin{rmrk}
The Riemannian Ricci tensor and the Riemannian scalar curvature of a general $G_2$ manifold are calculated in \cite{Br1}.
\end{rmrk}

\section{$G_2$-instanton. Proof of Theorem~\ref{instm}, Theorem~\ref{instqjm}  and Theorem~\ref{cmpga}}
We show the following
\begin{thrm}\label{maing2}
Let $(M,\p)$ be  a compact integrable $G_2$ manifold. The next two conditions are equivalent:
\begin{itemize}
\item[a)] The torsion 3-form is parallel with respect to the characteristic connection, $\sb T=0$.
\item[b)] The  curvature of the characteristic  connection $\sb$   is a $G_2$-instanton and $d^{\sb}T=0$.
\end{itemize}
\end{thrm}
\begin{proof}
If $\sb T=0$ then clearly $d^{\sb}T=\delta T=\sb\theta=d(Scal)=0$. 
Moreover, \eqref{4form} shows that the characteristic curvature $R\in S^2\Lambda^2$ and therefore $R$ is a $G_2$ instanton since $\sb\p=0$ which proves b).

For the converse, we first prove 
\begin{lemma}
If on an integrable $G_2$ manifold  the  curvature of the characteristic  connection $\sb$    is a $G_2$-instanton then the next equality holds true
\begin{equation}\label{bi24}
\sb_iR_{iplm}=\theta_rR_{rplm}.
\end{equation}
\end{lemma}
\begin{proof}
The second Bianchi identity for $\sb$ reads (see e.g. \cite{IS})
\begin{equation}\label{bi22}
\sb_iR_{jklm}+\sb_jR_{kilm}+\sb_kR_{ijlm}+T_{ijs}R_{sklm}+T_{jks}R_{silm}+T_{kis}R_{sjlm}=0.
\end{equation}
Multiplying \eqref{bi22} with $\ph_{ijkp}$ and using the $G_2$-instanton conditions \eqref{rri}, we obtain
\begin{equation}\label{bi23}
-6\sb_iR_{iplm}+3T_{ijs}R_{sklm}\ph_{ijkp}=0.
\end{equation}
An application of  \eqref{torcy2} together with \eqref{rri} to the second term in \eqref{bi23} yields
\begin{multline*}
T_{ijs}\ph_{ijkp}R_{sklm}=\Big[-T_{ijk}\ph_{ijps}-T_{ijp}\ph_{ijsk}-2T_{skp}-2\theta_r\ph_{rskp}+2\lambda\p_{skp}  \Big]R_{sklm}\\
=-T_{ijk}\ph_{ijps}R_{sklm}+2T_{ijp}R_{ijlm}-2T_{skp}R_{sklm}+4\theta_rR_{rplm}=-T_{ijs}\ph_{ijkp}R_{sklm}+4\theta_rR_{rplm}.
\end{multline*}
The last identity can be written as
\begin{equation}\label{m2}
T_{ijs}\ph_{ijkp}R_{sklm}=2\theta_rR_{rplm}.
\end{equation}
Substitute \eqref{m2} into \eqref{bi23} to get \eqref{bi24} which prooves the lemma.
\end{proof}
Let b) hods.  We multiply \eqref{bi24} with $T_{plm}$, using \eqref{1bi1}  the conditions $dT=2\sigma^T$,  $d^{\sb}T=0$ and the identity $\sigma^T_{ijkl}T_{ijk}=0$ proved in \cite{IS} to calculate 
\begin{multline}\label{bi25}
0=3\Big[\sb_iR_{iplm}-\theta_iR_{iplm}\Big]T_{plm}=\sb_i\Big[-\sigma^T_{plmi}+\sb_iT_{plm}\Big]T_{plm}
- \theta_i\Big[-\sigma^T_{plmi}+\sb_iT_{plm}\Big]T_{plm}\\
=-\sb_i\sigma^T_{plmi}T_{plm}+T_{plm}\sb_i\sb_iT_{plm} - \frac12\sb_{\theta}||T||^2=\sigma^T_{plmi}\sb_iT_{plm}+T_{plm}\sb_i\sb_iT_{plm} - \frac12\sb_{\theta}||T||^2\\=
\frac14\sigma^T_{plmi}d^{\sb}T_{iplm}+T_{plm}\sb_i\sb_iT_{plm} - \frac12\sb_{\theta}||T||^2=T_{plm}\sb_i\sb_iT_{plm} - \frac12\sb_{\theta}||T||^2.
\end{multline}
On the other hand, we calculate  the Laplacian  $-\Delta||T||^2=\LC_i\LC_i||T||^2=\sb_i\sb_i||T||^2$
\begin{equation}\label{lapt}
-\frac12\Delta||T||^2=T_{plm}\sb_i\sb_iT_{plm}+||\sb T||^2.
\end{equation}
A substitution of \eqref{lapt} into \eqref{bi25} yields
\begin{equation}\label{lap}
\Delta||T||^2 + \sb_{\theta}||T||^2=-2||\sb T||^2\le 0.
\end{equation}
Since $M$ is compact  we may apply  the strong maximum principle to \eqref{lap} (see e.g. \cite{YB,GFS}) to achieve  $\sb T=0$ which completes the proof of the theorem.
\end{proof}
As a consequence of the proof of Theorem~\ref{maing2}, we obtain from \eqref{lap}
\begin{cor}\label{main3}
Let $(M,\p)$ be  an integrable $G_2$ manifold. The next two conditions are equivalent:
\begin{itemize}
\item[a)] The torsion 3-form is parallel with respect to the characteristic connection, $\sb T=0$.
\item[b)] The  curvature of the characteristic  connection $\sb$   is a $G_2$-instanton, $d^{\sb}T=0$ and the norm of the torsion is constant, $d||T||^2=0$. 
\end{itemize}
\end{cor}
For completeness, we give the proof of the next observation from \cite{IS1}.
\begin{lemma}\cite[Lemma~5.1]{IS1}\label{inst1}
Let $(M,\p)$ be  an integrable $G_2$ manifold and the  curvature of the characteristic  connection $\sb$   is a $G_2$-instanton. Then  $\delta T\in\Lambda^2_{14}\cong  \frak{g}_2$.
\end{lemma}
\begin{proof}
Suppose the curvature $R$ of $\sb$ is a $G_2$-instanton. 
Multiply \eqref{gen} with  $\p$ and  apply  \eqref{rri} to get
\begin{equation}\label{rri2}
0=\Big[3R_{abci}-3R_{iabc}\Big]\p_{abc}=\Big[\frac32dT_{abci}-\sigma^T_{abci}\Big]\p_{abc}.
\end{equation}
We obtain from \eqref{rri2} and   \eqref{dtt} that
\begin{equation}\label{rri3}
12d\lambda_i=2\sb_iT_{abc}\p_{abc}=dT_{abci}\p_{abc}=\frac23\sigma^T_{abci}\p_{abc}.
\end{equation} 
Applying \eqref{lmt0} to \eqref{rri3}, we obtain
\begin{equation}\label{rri5}
\begin{split}
\sb_iT_{abc}\p_{abc}=6d\lambda_i=\frac13\sigma^T_{abci}\p_{abc}=-\theta_s\p_{sab}T_{abi}=-\theta_sT_{sab}\p_{abi}=d^{\sb}\theta_{ab}\p_{abi},
\end{split}
\end{equation}
where we used $d\theta\in\Lambda^2_{14}\cong\frak{g}_2$ and \eqref{ttg2} to achieve the last equality in \eqref{rri5}.

Substitute \eqref{rri5} into  \eqref{deltaT} to get
$
\delta T_{ab}\p_{abi}=0 \Leftrightarrow \delta T\in\Lambda^2_{14}\cong\frak{g}_2.
$
\end{proof}
\begin{lemma}\label{ggg}
Let on  an integrable $G_2$ manifold with $d^{\sb}\theta=0$  the characteristic  curvature   is a $G_2$-instanton. Then $\delta T=0$, the manifold is of constant type and the characteristic Ricci tensor is symmetric.
\end{lemma}
\begin{proof}
The condition $d^{\sb}\theta=0$ and \eqref{deltaT} imply
$
\delta T=-d\lambda\lrcorner\p \in\Lambda^2_7.
$ 
Hence $\delta T=d\lambda=0$ by Lemma~\ref{inst1}.
\end{proof}
\subsection{Proof of Theorem~\ref{instm}}
\begin{proof}
Suppose $\sb T=0$. Then $d^{\sb}T=\delta T=0$ and \eqref{dh} implies \eqref{inst6}. Therefore  $\sb dT=2\sb\sigma^T=0$ the Ricci tensor of the torsion connection is symmetric, because of \eqref{rics},  $\sb$-parallel with constant scalar curvature. Moreover, \eqref{4form} shows that the characteristic curvature $R\in S^2\Lambda^2$ and therefore $R$ is a $G_2$ instanton since $\sb\p=0$.

To prove the converse, we begin with the following
\begin{prop}\label{instm1}
Let $(M,\p)$ be  an integrable $G_2$ manifold with $\sb$-parallel Lee form and the  curvature of the characteristic  connection $\sb$   is a $G_2$-instanton. Then $d^{\sb}T=0$.
 \end{prop}
 \begin{proof}
 Since $d\lambda=0$ due to Lemma~\ref{ggg}, it follows from \eqref{dtt} and \eqref{rri2} 
 $0=dT_{iabc}\p_{abc},\quad \sigma^T_{iabc}\p_{abc}=0$  and   \eqref{dh} yields
 \begin{equation}\label{inst3}
 0=dT_{iabc}\p_{abc}=d^{\sb}T_{iabc}\p_{abc}+2\sigma^T_{iabc}\p_{abc}=d^{\sb}T_{iabc}\p_{abc}.
 \end{equation}
  Further we use the $G_2$-instanton condition \eqref{rri}. 
Multiply \eqref{gen} with $\ph$ and use \eqref{rri} to get
\begin{equation}\label{rri21}
\Big[3R_{abci}-3R_{iabc}\Big]\ph_{abcj}=-6R_{cjci}+6R_{iaaj}=6Ric_{ji}+6Ric_{ij}=\Big[\frac32dT_{abci}-\sigma^T_{abci}\Big]\ph_{abcj}.
\end{equation}
We obtain from \eqref{rri21}, \eqref{ricdt} using \eqref{dh}, \eqref{deltaT} and $d\lambda=0$ that
\begin{equation*}
\begin{split}
-\delta T_{ij}=-d^{\sb}\theta_{ij}=Ric_{ij}-Ric_{ji}=
\frac16\Big[-\frac12dT_{abci}-2\sb_iT_{abc}+\sigma^T_{abci} \Big]\ph_{abcj}\\=
-\frac14\Big[\sb_aT_{bci}+\sb_iT_{abc} \Big]\ph_{abcj}=-\frac14\sb_aT_{bci}\ph_{abcj}-\frac32\sb_i\theta_j ,
\end{split}
\end{equation*}
which implies
\begin{equation}\label{g22}
\begin{split}
\sb_aT_{bci}\ph_{abcj}=-6\sb_i\theta_j+4\sb_i\theta_j-4\sb_j\theta_i=-2\sb_i\theta_j-
4\sb_j\theta_i . 
\end{split}
\end{equation}
 Now, \eqref{g22} and $\sb\theta=0$ yield 
 \begin{equation}\label{inst4}
 \sb_aT_{bci}\ph_{abcj}=0.
 \end{equation}
 Substitute \eqref{inst4} into \eqref{dh} and use again  $\sb\theta=0$ to get
 \begin{equation}\label{inst5}
 d^{\sb}T_{iabc}\ph_{abcj}=0.
 \end{equation}
 Hence, \eqref{inst3} and \eqref{inst5} imply that  for any $X\in T_pM$ the 3-form  $X\lrcorner d^{\sb}T\in \Lambda^3_{27}$ and Proposition~\ref{4-for} yields  $ d^{\sb}T=0$ and $dT=2\sigma^T$ by   \eqref{dh}. 
 \end{proof}
 To handle the non-compact case, we observe
 \begin{prop}\label{instm2}
Let $(M,\p)$ be  an integrable $G_2$ manifold with $\sb$-parallel Lee form and the  curvature of the characteristic  connection $\sb$   is a $G_2$-instanton. Then the norm of the torsion is a constant, $d||T||^2=0$.\
 \end{prop}
 \begin{proof}
 It is known due to \cite[(3.38)]{I} that 
\begin{multline}\label{ins1}
2R(X,Y,Z,V)-2R(Z,V,X,Y)\\=(\sb_XT)(Y,Z,V)-(\sb_YT)(X,Z,V)-(\sb_ZT)(X,Y,V)+(\sb_VT)(X,Y,Z).
\end{multline}
 Proposition~\ref{instm1} tells us that \eqref{inst6} holds true.   Using $d^{\sb}T=0$, we obtain from \eqref{ins1} that
 \begin{equation}\label{ins2}
 R_{ijkl}-R_{klij}=\sb_iT_{jkl}-\sb_jT_{ikl}=-\sb_kT_{lij}+\sb_lT_{kij}. 
 \end{equation}
 Multiply \eqref{ins2}with $\ph_{ijab}$, use the instanton condition, \eqref{inst6} and \eqref{ins1} to get
 \begin{multline}\label{ins3}
 2\sb_iT_{jkl}\ph_{ijab}=\Big[-\sb_kT_{lij}+\sb_lT_{kij}\Big]\ph_{ijab}\\=-2R_{abkl}+2R_{klab}=2\Big[\sb_kT_{lab}-\sb_lT_{kab}\Big]=-2\Big[\sb_aT_{bkl}-\sb_bT_{akl}\Big].
 \end{multline}
 We will use the contracted second Bianchi identity for a metric connection with totally skew-symmetric torsion proved in  \cite[Proposition~3.5]{IS}
\begin{equation}\label{e1}
d(Scal)_j-2\sb_iRic_{ji}+\frac16d||T||^2_j+\delta T_{ab}T_{abj}+\frac16T_{abc}dT_{jabc}=0.
\end{equation}
We obtain from \eqref{inst6} that
\begin{equation}\label{inst8}
\begin{split}
0=d^{\sb}T_{absi}T_{abs}=3\sb_aT_{bsi}T_{abs}-\sb_iT_{abs}T_{abs}=3\sb_aT_{bsi}T_{abs}-\frac12\sb_i||T||^2;\\
0=d^{\sb}T_{absi}\sb_iT_{abs}=3\sb_aT_{bsi}\sb_iT_{abs}-\sb_iT_{abs}\sb_iT_{abs}=3\sb_aT_{bsi}\sb_iT_{abs}-||\sb T||^2.
\end{split}
\end{equation}
Further, we get  from \eqref{ricg2} applying \eqref{inst6} and the condition $\sb\theta=0$ that
 \begin{equation}\label{ins4}
 Ric_{ij}=\frac1{12}dT_{iabc}\ph_{jabc}=\frac16\sigma^T_{iabc}\ph_{jabc}=-\frac12T_{abs}T_{sci}\ph_{jabc}.
 \end{equation}
 We calculate from \eqref{ins4} using \eqref{ins3}, \eqref{inst4} and \eqref{inst8} that
 \begin{multline}\label{inst7}
 -2\sb_jRic_{ij}=\sb_jT_{abs}\ph_{jabc}T_{sci}+T_{abs}\sb_jT_{cis}\ph_{jcab}=-T_{abs}\Big[\sb_aT_{bis}-\sb_bT_{ais}\Big]=\frac13\sb_i||T||^2.
 \end{multline}
 We obtain from \eqref{ricg2} using $\sb\theta=d\lambda=0$ 
 \begin{equation}\label{inst9}
 d(Scal)_j=-\frac13\sb_j||T||^2.
 \end{equation}
 Substitute \eqref{inst7}and \eqref{inst9} into \eqref{e1} to get
 $$d||T||^2=0,$$
 where we used  $\delta T=0$ and the identity $dT_{jabc}T_{abc}=2\sigma^T_{jabc}T_{abc}=0$ proved in  \cite[Proposition~3.1]{IS}.
 \end{proof}
  Combine Corollary~\ref{main3} with Proposition~\ref{instm1} and Proposition~\ref{instm2}  to complete the proof of Theorem~\ref{instm}.
 \end{proof}

Since on a co-calibrated $G_2$ manifold the Lee form $\theta=0$, we obtain from Proposition~\ref{instm1} and Proposition~\ref{instm2}  that
 \begin{cor}
 Let $(M,\p)$ be  a co-calibrated $G_2$ manifold  and the  curvature of the characteristic  connection $\sb$   is a $G_2$-instanton. Then \eqref{inst6} holds true.
 \end{cor}
 \begin{cor}\label{co-inst}
Let $(M,\p)$ be  a co-calibrated $G_2$ manifold and the  curvature of the characteristic  connection $\sb$   is a $G_2$-instanton. Then the norm of the torsion is a constant, $d||T||^2=0$.
 \end{cor}

 \subsection{Proof of Theorem~\ref{instqjm}}
 The statements in Theorem~\ref{instqjm} follow from \cite[Theorem~1.1, Lemma~5.1]{IS1} and Theorem~\ref{instm}. Indeed, \cite[Lemma~5.1]{IS1} shows that $\sb\theta=0$ and  Theorem~\ref{instm} yields $\sb T=0$. Now \cite[Theorem~1.1]{IS1} implies $dT=\delta T=\LC T=\sb T=0$. Therefore, $\sigma^T=0$ by \eqref{dh}.

 \subsection{Compact  Gauduchon $G_2$ manifolds. Proof of Theorem~\ref{cmpga}} 
In this subsection, we recall the notion of conformal deformations of a given $G_2$ structure $\p$  from \cite{FG,FI1,Kar1} and proof Theorem~\ref{cmpga}. 

Let $\bar\p=e^{3f}\p$ be a conformal deformation of $\p$. The induced metric $\bar g=e^{2f}g$ and $\bar{*}\bar\p=e^{4f}*\p$, where $\bar{*}$ is the Hodge star operator with respect to $\bar{g}$. The class of integrable $G_2$ structures is invariant under conformal deformations. An easy calculations give $(d\bar\p,\bar{*}\bar\p)=e^{-f}(d\p,*\p)$ which compared with \eqref{lm} yields $\bar\lambda=e^{-f}\lambda$. Hence, the class of strictly integrable $G_2$ manifolds, ($\lambda=0$), is invariant under conformal deformations while the class of constant non-zero type is not conformally invariant.

The Lee forms are connected by $\bar\theta=\theta+4df$. Using the expression of the Gauduchon theorem 
 in terms of a Weyl structure \cite [Appendix 1]{tod},
one can find, in a unique way, a conformal $G_2$ structure such that the corresponding Lee 
1-form is coclosed with respect to the induced metric due to \cite[Theorem~3.1]{FI1}.

Further, we establishe the following
\begin{thrm}\label{RG1}
Let $(M,\p)$ be a compact integrable $G_2$ manifold of constant type with a Gauduchon $G_2$ structure, $\delta\theta=0$.  If the characteristic connection is a $G_2$-instanton then the Lee form is $\sb$-parallel. 

In particular  $\delta T=0$ and the Ricci tensor is symmetric.
\end{thrm}
\begin{proof} We start with the next identity 
\begin{equation}\label{iii}
\sb_i\delta T_{ij}=\frac12\delta T_{ia}T_{iaj}.
\end{equation}
 shown in \cite[Proposition~3.2]{IS} for any metric connection with a totally skew-symmetric torsion. 

We calculate the left-hand side  of \eqref{iii} applying \eqref{deltaT} as follows
\begin{equation}\label{ntss}
\begin{split}
\sb_i\delta T_{ij}=\sb_i[d^{\sb}\theta_{ij}-\sb_t\lambda\p_{tij}]=
\sb_i\sb_i\theta_j-\sb_i\sb_j\theta_i-\frac12T_{tis}\sb_s\lambda\p_{tij},
\end{split}
\end{equation}
where we applied $d^2\lambda=0$ and \eqref{tsym} to get the last term.

Substitute \eqref{ntss} into \eqref{iii} using \eqref{deltaT} to get
\begin{equation}\label{ntss1}
\sb_i\sb_i\theta_j-\sb_i\sb_j\theta_i-\frac12T_{abs}\sb_s\lambda\p_{sab}=\frac12d^{\sb}\theta_{ab}T_{abj}-\frac12T_{abj}\sb_s\lambda\p_{sab}.
\end{equation}
The Ricci identity  
\begin{equation}\label{rrii}\sb_i\sb_j\theta_i=\sb_j\sb_i\theta_i-R_{ijis}\theta_s-T_{ija}\sb_a\theta_i=
\sb_j\sb_i\theta_i+Ric_{js}\theta_s-\frac12d^{\sb}\theta_{ai}T_{aij}
\end{equation}
substituted into \eqref{ntss1} yields
\begin{equation}\label{gafin}
\sb_i\sb_i\theta_j+\sb_j\delta\theta-Ric_{js}\theta_s=\frac12\sb_s\lambda\Big(T_{abs}\p_{abj}-T_{abj}\p_{abs}  \Big)=-\sb_s\lambda\theta_a\p_{asj},
\end{equation}
where we use the first identity of \eqref{lmt0} to achieve the last equality. 

Multiply the both sides of \eqref{gafin} with $\theta_j$, use  $\delta\theta=0$ together with the identity 
\begin{equation}\label{llii}\frac12\Delta||\theta|^2=-\frac12\LC_i\LC_i||\theta||^2=-\frac12\sb_i\sb_i||\theta||^2=-\theta_j\sb_i\sb_i\theta_j-||\sb\theta||^2
\end{equation}
to get (see \cite{IS1})
\begin{equation}\label{gafinf}
-\frac12\Delta||\theta||^2-Ric(\theta,\theta)-||\sb\theta||^2=0.
\end{equation}
 Since $d\lambda=0$ we have from \eqref{deltaT} that $\delta T=d^{\sb}\theta$. Consequently, \eqref{g22} holds true. We calculate from \eqref{ricg2} with the help of \eqref{g22}  that
\begin{equation}\label{rri4}
\begin{split}
Ric_{ij}\theta_i\theta_j=\frac1{12}dT_{abci}\ph_{abcj}\theta_i\theta_j-\theta_i\theta_j\sb_i\theta_j\\=
\frac1{12}\Big[2\sigma^T_{abci}\ph_{abcj}+3\sb_aT_{bci}\ph_{abcj}-18\sb_i\theta_j  \Big]\theta_i\theta_j=-\frac32\theta_i\sb_i||\theta||^2,
\end{split}
\end{equation}
where we used $\theta\lrcorner T=d\theta-\delta T\in \Lambda^2_{14}\cong\frak{g}_2$ due to Lemma~\ref{inst1}, to get  $\sigma^T_{abci}\ph_{abcj}\theta_i\theta_j=0$. 

Indeed,  we calculate applying \eqref{torcy2} and $\theta\lrcorner T\in \Lambda^2_{14}\cong\frak{g}_2$
\begin{multline*}
\frac13\sigma^T_{jsmp}\ph_{jsmk}\theta_p =T_{jsl}T_{lmp}\ph_{jsmk}\theta_p
=-T_{klm}T_{lmp}\theta_p-\frac12T_{jsk}\ph_{jslm}T_{lmp}\theta_p\\-\theta_a\ph_{aklm}T_{lmp}\theta_p+\lambda\p_{klm}T_{lmp}\theta_p
=-T_{klm}T_{lmp}\theta_p+T_{jsk}T_{jsp}\theta_p+2\theta_aT_{akp}\theta_p=0.
\end{multline*}
Substitute \eqref{rri4} into \eqref{gafinf} to obtain 
\begin{equation}\label{gafinf1}
\Delta||\theta||^2+3\theta_i\sb_i||\theta||^2=-2||\sb\theta||^2\le 0.
\end{equation}
W  apply  the strong maximum principle to \eqref{gafinf1} (see e.g. \cite{YB,GFS}) to achieve $d||\theta||^2=\sb\theta=0$.
\end{proof}
The proof of Theorem~\ref{cmpga} follows from Theorem~\ref{RG1}, Proposition~\ref{instm1} and Theorem~\ref{maing2}.

\subsection{Closed torsion}
If the torsion is closed then the integrable $G_2$ manifold is said to be \emph{strong}.
\begin{thrm}\label{g2clost}
Let $(M,\p)$ be a  compact  integrable $G_2$ manifold with  closed torsion. 

The next conditions are equivalent:
\begin{itemize}
\item[a)] The torsion is $\sb$-parallel.
\item[b)] The curvature of $\sb$ is a $G_2$-instanton.
\end{itemize}
\end{thrm}
\begin{proof}
If $\sb T=0$ then $R\in S^2\Lambda^2$ by \cite[Lemma~3.4]{I} which shows that $R$ is a $G_2$-instanton.

For the converse, let $R$ be a $G_2$-instanton. Then \eqref{bi24} holds true.

We multiply \eqref{bi24} with $T_{plm}$, using \eqref{1bi1}  and $dT=0$ to calculate 
\begin{multline}\label{bi25clg}
0=3\Big[\sb_iR_{iplm}-\theta_iR_{iplm}\Big]T_{plm}=T_{plm}\sb_i\sb_iT_{plm}-\theta_i\sb_iT_{plm}T_{plm}\\=T_{plm}\sb_i\sb_iT_{plm}-\frac12\sb_{\theta}||T||^2.
\end{multline}
A substitution of \eqref{lapt} into \eqref{bi25clg} yields that \eqref{lap} holds true.
Since $M$ is compact  we may apply  the strong maximum principle to \eqref{lap} (see e.g. \cite{YB,GFS}) to achieve  $\sb T=0$ which completes the proof of the theorem.
\end{proof}

\section{$Spin(7)$-structure}

We briefly recall the notion of a $Spin(7)$-structure. Consider
${\mathbb R}^8$ endowed with an orientation and its standard inner
product. Consider the 4-form $\ps$ on ${\mathbb R}^8$ given by
\begin{eqnarray}\label{s1}
\ps &=&-e_{0127} +e_{0236} - e_{0347}-e_{0567} +e_{0146} + e_{0245} -
  e_{0135}
 \\ \nonumber&\phantom{=}&
-e_{3456} - e_{1457} - e_{1256}-e_{1234} - e_{2357} -
  e_{1367} +e_{2467}.\nonumber
\end{eqnarray}
The 4-form  $\ps$ is self-dual, $*\ps=\ps$, and the 8-form $\ps\wedge\ps$ coincides with
 14 times the volume form of ${\mathbb R}^8$. The subgroup of $GL(8,\mathbb R)$ which
fixes $\ps$ is isomorphic to the double covering $Spin(7)$ of
$SO(7)$ \cite{Br}. Moreover, $Spin(7)$ is a compact
simply-connected Lie group of dimension 21 \cite{Br}. The Lie algebra of $Spin(7)$ is
denoted by $\frak{spin}(7)$ and it is isomorphic to the 2-forms satisfying  linear equations, namely
 $\frak{spin}(7)\cong  \{\alpha \in
\Lambda^2(M)|*(\alpha\wedge\ps)=\alpha\}$. We note here the sign difference with \cite{Br}.

The 4-form
$\ps$ corresponds to a real spinor $\ps$ and therefore,
$Spin(7)$ can be identified as the isotropy group of a non-trivial
real spinor.

We let the expression
$$
 \ps=\frac1{24}\ps_{ijkl}e_{ijkl}
$$
and thus have the  identites (c.f.  \cite{GMW,Kar2})
\begin{eqnarray}\nonumber
\ps_{ijpq}\ps_{ijpq} & = & 336;
\\\label{p1}%
\ps_{ijpq}\ps_{ajpq} &=& 42\delta_{ia};
\\\nonumber
\ps_{ijpq}\ps_{klpq} &=& 6\delta_{ik}\delta_{jl} -
  6\delta_{il}\delta_{jk} - 4\ps_{ijkl} ;\\\nonumber
\ps_{ijks}\ps_{abcs} &=& \delta_{ia} \delta_{jb} \delta_{kc} +\delta_{ib} \delta_{jc} \delta_{ka} +\delta_{ic} \delta_{ja} \delta_{kb}\\\nonumber
&-& \delta_{ia} \delta_{jc} \delta_{kb} -\delta_{ib} \delta_{ja} \delta_{kc} -\delta_{ic} \delta_{jb} \delta_{ka} \\\nonumber
&-& \delta_{ia}\ps_{jkbc}-\delta_{ja}\ps_{kibc}-\delta_{ka}\ps_{ijbc}\\\nonumber
&-& \delta_{ib}\ps_{jkca}-\delta_{jb}\ps_{kica}-\delta_{kb}\ps_{ijca}\\\nonumber
&-& \delta_{ic}\ps_{jkab}-\delta_{jc}\ps_{kiab}-\delta_{kc}\ps_{ijab}.
\end{eqnarray}
A \emph{$Spin(7)$-structure} on an 8-manifold $M$ is by definition
a reduction of the structure group of the tangent bundle to
$Spin(7)$; we shall also say that $M$ is a \emph{$Spin(7)$-manifold}. This can be described geometrically by saying that
there exists a nowhere vanishing global differential 4-form $\ps$
on $M$ which can be locally written as \eqref{s1}. The 4-form
$\ps$ is called the \emph{fundamental form} of the $Spin(7)$-manifold $M$ \cite{Bo}.  Alternatively, a $Spin(7)$-structure can be described by the existence of three-fold
vector cross product  on the tangent spaces of $M$ (see e.g. \cite{Gr}).

The fundamental form of a $Spin(7)$-manifold determines a
Riemannian metric  $g$ which 
  is referred  tp as the   metric induced by $\ps$. We write $\LC$ for the associated Levi-Civita
connection and  $||.||^2$ for the tensor norm with respect to $g$. 

In addition, we will freely identify
vectors and co-vectors via the induced metric $g$.

In general, not every  compact 8-dimensional Riemannian spin manifold $M^8$
admits a $Spin(7)$-structure. We explain the precise condition
\cite{LM}. Denote by $p_1(M), p_2(M), {\mathbb X}(M), {\mathbb
X}(S_{\pm})$ the first and the second Pontrjagin classes, the
Euler characteristic of $M$ and the Euler characteristic of the
positive and the negative spinor bundles, respectively. It is well
known \cite{LM} that a compact spin 8-manifold admits a $Spin(7)$-structure if and only if ${\mathbb X}(S_+)=0$ or ${\mathbb X}(S_-)=0$.
The latter conditions are equivalent to $
p_1^2(M)-4p_2(M)+ 8{\mathbb X}(M)=0$, for an appropriate choice
of the orientation \cite{LM}.

Let us recall that a $Spin(7)$-manifold $(M,g,\ps)$ is said to be
parallel (torsion-free) if the holonomy $Hol(g)$ of the metric $g$ is a subgroup of $Spin(7)$. This is equivalent to saying
that the fundamental form $\ps$ is parallel with respect to the
Levi-Civita connection of the metric $g$, $\nabla^g\ps=0$.  

M. Fernandez shows in \cite{Fer} that
$Hol(g)\subset Spin(7)$ if and only if $d\ps=0$ which is equivalent to $\delta\ps=0$ since $\ps$ is self-dual 4-form  (see also \cite{Br,Sal}).  It was observed  by Bonan that any parallel $Spin(7)$-manifold is Ricci flat
\cite{Bo}. The first known explicit example of complete parallel $Spin(7)$-manifold with $Hol(g)=Spin(7)$ was constructed by Bryant and
Salamon \cite{BS,Gibb}.
The first compact examples of parallel $Spin(7)$-manifolds with
$Hol(g)=Spin(7)$ were constructed by Joyce \cite{J1,J2}.

There are 4 classes of $Spin(7)$-manifolds according to the
Fernandez classification \cite{Fer} obtained as irreducible $Spin(7)$
representations of  the space $\nabla^g\ps$.

The Lee form $\theta$ is defined by \cite{C1}
\begin{equation}\label{sg2li}
\theta = -\frac{1}{7}*(*d\ps\wedge\ps)=\frac17*(\delta\ps\wedge
\ps)=\frac1{7}(\delta\ps)\lrcorner\ps,\quad \theta_a=\frac1{42}(\delta\ps)_{ijk}\ps_{ijka},
\end{equation}
where  $\delta=-*d*$ is the co-differential acting on $k$-forms in dimension eight.

The 4 classes of Fernandez classification \cite{Fer} can be described in
terms of the Lee form as follows \cite{C1}: $W_0 : d\ps=0; \quad
W_1 : \theta =0; \quad W_2 : d\ps = \theta\wedge\ps; \quad W :
W=W_1\oplus W_2.$

A $Spin(7)$-structure of the class $W_1$ (i.e.
$Spin(7)$-structure with zero Lee form) is called
 \emph{a balanced $Spin(7)$-structure}.
If the Lee form is closed, $d\theta=0,$ then the $Spin(7)$-structure is
locally conformally equivalent to a balanced one \cite{I1} (see also \cite{Kar1,Kar2}).
It is known due to  \cite{C1} that the Lee form of a $Spin(7)$-structure in the class $W_2$ is closed and therefore such a
manifold is locally conformally equivalent to a parallel $Spin(7)$-manifold. 

If $M$ is compact then it is shown in \cite[Theorem~4.3]{I1} that  in every conformal class of  $Spin(7)$-structures $[\ps]$ there exists a unique $Spin(7)$-structure with co-closed Lee form, $\delta\theta=0$. The compact $Spin(7)$-spaces with closed but not exact Lee form
(i.e. the structure is not globally
conformally parallel) have very different topology than the parallel ones
\cite{I1,IPP}.

Coeffective cohomology and coeffective numbers of 
a $Spin(7)$ manifold are studied in \cite{Ug}.

\subsection{Decomposition of the space of forms} We take the following description of the decomposition of the space of forms from \cite{Kar2}.

Let $(M, \ps)$ be a $Spin(7)$-manifold. The action of $Spin(7)$  on the tangent space induces an
action of $Spin(7)$ on $\Lambda^k(M)$ splitting the exterior algebra into orthogonal irreducible $Spin(7)$  subspaces, where
$\Lambda^k_l$ corresponds to an $l$-dimensional $Spin(7)$-irreducible subspace of $\Lambda^k$:
\begin{equation*}
\begin{split}
\Lambda^2(M)=\Lambda^2_7\oplus\Lambda^2_{21}, \qquad \Lambda^3(M)=\Lambda^3_8\oplus\Lambda^3_{48},\qquad
\Lambda^4(M)=\Lambda^4_1\oplus\Lambda^4_7\oplus\Lambda^4_{27}\oplus\Lambda^4_{35},
\end{split}
\end{equation*}
where
\begin{equation}\label{sdec2}
\begin{split}
\Lambda^2_7=\{\alpha\in \Lambda^2(M) | *(\alpha\wedge\ps)=-3\alpha\};\\
\Lambda^2_{21}=\{\alpha\in \Lambda^2(M) | *(\alpha\wedge\ps)=\alpha\}\cong \frak{spin}(7);\\
\Lambda^3_8=\{*(\gamma\wedge\ps) |  \gamma\in\Lambda^1\}=\{\gamma\lrcorner\ps\};\\
\Lambda^3_{48}=\{\gamma\in \Lambda^3(M) | \gamma\wedge\ps=0\}.
\end{split}
\end{equation}
Hence, a 2-form $\ps$ decomposes into two $Spin(7)$-invariant parts, $\Lambda^2=\Lambda^2_7\oplus\Lambda^2_{21}$, and
\begin{equation*}
\begin{split}
\alpha\in \Lambda^2_7 \Leftrightarrow \alpha_{ij}\ps_{ijkl}=-6\alpha_{kl},\\
\alpha\in \Lambda^2_{21} \Leftrightarrow \alpha_{ij}\ps_{ijkl}=2\alpha_{kl}.
\end{split}
\end{equation*}

For $k>4$ we have $\Lambda^k_l=*\Lambda^{8-k}_l$.

For $k=4$, following \cite{Kar2}, one considers the operator $\Omega_{\ps}:\Lambda^4 \longrightarrow\Lambda^4$ defined as follows
\begin{equation}\label{op}
\begin{split}
(\Omega_{\ps}(\sigma))_{ijkl}=\sigma_{ijpq}\ps_{pqkl}+\sigma_{ikpq}\ps_{pqlj}+\sigma_{ilpq}\ps_{pqjk}+\sigma_{jkpq}\ps_{pqil}+\sigma_{jlpq}\ps_{pqki}+\sigma_{klpq}\ps_{pqij}.
\end{split}
\end{equation}
\begin{prop}\cite[Proposition~2.8]{Kar2}\label{427}
The spaces $\Lambda^4_1,\Lambda^4_7,\Lambda^4_{27},\Lambda^4_{35}$ are all eigenspaces of the operator $\Omega_{\ps}$ with distinct eigenvalues. Specifically,
\begin{equation}\label{dec4}
\begin{split}
\Lambda^4_1=\{\sigma\in\Lambda^4:\Omega_{\ps}(\sigma)=-24\sigma\};\qquad \Lambda^4_7=\{\sigma\in\Lambda^4:\Omega_{\ps}(\sigma)=-12\sigma\};\\\Lambda^4_{27}=\{\sigma\in\Lambda^4:\Omega_{\ps}(\sigma)=4\sigma\}=\{\sigma\in\Lambda^4:\sigma_{ijkl}\ps_{mjkl}=0\};\qquad \Lambda^4_{35}=\{\sigma\in\Lambda^4:\Omega_{\ps}(\sigma)=0\};\\
\Lambda^4_+=\{\sigma\in\Lambda^4:*\sigma=\sigma\}=\Lambda^4_1\oplus\Lambda^4_7\oplus\Lambda^4_{27};\qquad \Lambda^4_-=\{\sigma\in\Lambda^4:*\sigma=-\sigma\}=\Lambda^4_{35}.
\end{split}
\end{equation}
\end{prop}

We recall the following
\begin{dfn}
The curvature $R$ of a linear connection on a $Spin(7)$ manifold  is a $Spin(7)$-instanton if  the curvature 2-form lies in the lie algebra $\frak{spin}(7)\cong\Lambda^2_{21}$. This is equivalent to  the  identity:
\begin{equation}\label{srri}
 R_{abij}\ps_{abkl}=2R_{klij}.
\end{equation}
\end{dfn}

\subsection{The $Spin(7)$-connection with skew-symmetric torsion}
The presence of a parallel spinor with respect to a metric connection with torsion 3-form leads to the reduction of the holonomy group of the torsion connection to a subgroup of $Spin(7)$. 
It is shown in \cite{I1} that any $Spin(7)$-manifold $(M,\ps)$ admits a unique $Spin(7)$-connection with 
torsion 3-form. 
\begin{thrm}\cite[Theorem~1]{I1}
Let $(M,\ps)$  be a $Spin(7)$-manifold with fundamental 4-form $\ps$. There always exists a unique linear connection $\sb$ preserving the $Spin(7)$-structure, $\sb\ps=\sb g=0,$ with totally skew-symmetric torsion $T$ given by
\begin{equation}\label{storcy1}
T=-*d\ps+\frac76*(\theta\wedge\ps)=\delta\ps+\frac76\theta\lrcorner\ps,
\end{equation}
where the Lee form $\theta$ is given by \eqref{sg2li}.
\end{thrm}
Note that we use here $\ps:=-\ps$ from \cite{I1}.  

See also \cite{Fr,Mer} for subsequent proofs of this theorem.

Express the codifferential of the 4-form $\ps$ in terms of the Levi-Civita connection and then in terms of the torsion connection using \eqref{tsym}, \eqref{p1}, \eqref{storcy1} and $\sb\ps=0$ to get the next formulas presented in \cite{Iap}
\begin{equation}\label{storcy2}
\begin{split}
T_{klm}=\frac12T_{jsk}\ps_{jslm}+\frac12T_{jsl}\ps_{jsmk}+\frac12T_{jsm}\ps_{jskl}+\frac76\theta_s\ps_{sklm}, \quad \theta_i=-\frac17T_{jkl}\ps_{jkli}.
\end{split}
\end{equation}
Denote the skew-symmetric part of $\sb\theta$ by $d^{\sb}\theta$, $d^{\sb}\theta_{ij}=\sb_i\theta_j-\sb_j\theta_i$, we express the co-differential of the torsion  with  the next formula from \cite{Iap}
\begin{equation}\label{sdeltaT}
\delta T=\frac76(d\theta\lrcorner\ps-\theta\lrcorner T)=\frac76\Big(d^{\sb}\theta\lrcorner \ps+(\theta\lrcorner T)\lrcorner\ps-\theta\lrcorner T\Big).
\end{equation}

The Ricci tensor $Ric$ and the scalar curvature $Scal$ of the torsion connection were calculated in \cite{I1} with the help of the properties of the $\sb$-parallel real spinor corresponding to the $Spin(7)$ form $\ps$, applying the Schr\"odinger-Lichnerowicz formula for the torsion connection as follows (see also \cite{Iap}) 
\begin{equation}\label{sricg2}
\begin{split}
Ric_{ij}=-\frac1{12}dT_{iabc}\ps_{jabc}-\frac76\sb_i\theta_j;\qquad
Scal=\frac72\delta\theta+\frac{49}{18}||\theta||^2-\frac13||T||^2.
\end{split}
\end{equation}
\begin{rmrk}\label{remsp7}
It follows from \eqref{sricg2}, \eqref{dh}, \eqref{sigma} and \eqref{inst6} that if  $\sb T=0$ then $\delta T=  \sb Ric=\sb dT=0$ and $d(Scal)=0$.
\end{rmrk}
\section{$Spin(7)$-instanton. Proof of Theorem~\ref{main2}, Theorem~\ref{mainsp0}  and Theorem~\ref{mainsp}}
Since $Hol(\sb)\in\frak{spin}(7)\cong\Lambda^2_{21}$, we have from the first Bianchi identity \eqref{1bi} applying \eqref{dh} that
\begin{equation}\label{spbi1} 
R_{ijkl}\ps_{ijkm}=-2Ric_{ml}=\frac13\Big[ d^{\sb}T_{ijkl}+\sigma^T_{ijkl}+\sb_lT_{ijk}\Big]\ps_{ijkm}.
\end{equation}
Note that \eqref{spbi1} is equivalent to the first equation in \eqref{sricg2}.

The $Spin(7)$-instanton condition \eqref{srri} together with \eqref{1bi1} and \eqref{dh} imply
\begin{equation}\label{spbi2}
R_{lijk}\ps_{ijkm}=2Ric_{lm}=\frac13\Big[-\frac12 d^{\sb}T_{ijkl}-\sigma^T_{ijkl}+\sb_lT_{ijk}\Big]\ps_{ijkm}.
\end{equation}
\begin{prop}\label{spin7}
Let $(M,\ps)$ be  a $Spin(7)$ manifold and the  curvature of the torsion  connection $\sb$   is a $Spin(7)$-instanton.  Then the following  hold true 
\begin{equation}\label{in1}\delta T\in\Lambda^2_{21}\cong\frak{spin}(7); \quad 3d^{\sb}\theta+4\theta\lrcorner T=3d\theta+\theta\lrcorner T\in\Lambda^2_{21}\cong\frak{spin}(7).
\end{equation} 
\begin{itemize}
\item[a)] If  $d\theta=0$ then 
\begin{equation}\label{sdeltat1}
\begin{split}
d^{\sb}\theta=-\theta\lrcorner T\in\Lambda^2_{21}\cong\frak{spin}(7),\quad d^{\sb}T_{ijkl}\ps_{ijkm}=d^{\sb}T_{ijkm}\ps_{ijkl}, \quad \delta T=\frac76d^{\sb}\theta=-\frac76\theta\lrcorner T.
\end{split}
\end{equation}
\item[b)] If $d^{\sb}\theta=0$ then 
\begin{equation}\label{sdeltat12}
\begin{split}
d\theta=\theta\lrcorner T\in\Lambda^2_{21}\cong\frak{spin}(7),\quad d^{\sb}T_{ijkl}\ps_{ijkm}=d^{\sb}T_{ijkm}\ps_{ijkl}, \quad \delta T=0.
\end{split}
\end{equation}
\end{itemize}
\end{prop}
\begin{proof}
The sum of \eqref{spbi1} and \eqref{spbi2} gives applying \eqref{rics}
\begin{equation}\label{spin1}
\delta T_{ml}=Ric_{lm}-Ric_{ml}=\frac1{12}d^{\sb}T_{ijkl}\ps_{ijkm}-\frac73\sb_l\theta_m=\frac14\Big[\sb_iT_{jkl}+\sb_lT_{ijk}\Big]\ps_{ijkm}.
\end{equation}
On the other hand, we obtain  after taking the trace of the covariant derivative of \eqref{storcy2}
\begin{equation}\label{spin2}
\begin{split}
-2\delta T_{lm}+\delta T_{js}\ps_{jslm}=-\sb_kT_{jsl}\ps_{kjsm}+\sb_kT_{jsm}\ps_{kjsl}+\frac73\sb_k\theta_s\ps_{sklm}\\
=\frac13\Big[ \sb_kT_{jsm}+\sb_jT_{skm}+\sb_sT_{kjm}\Big]\ps_{kjsl}-\frac13\Big[ \sb_kT_{jsl}+\sb_jT_{skl}+\sb_sT_{kjl}\Big]\ps_{kjsm}+\frac73\sb_k\theta_s\ps_{sklm}\\
=\frac13\Big[d^{\sb}T_{kjsm}\ps_{kjsl}-d^{\sb}T_{kjsl}\ps_{kjsm}\Big]+\frac73d^{\sb}\theta_{lm}-\frac76d^{\sb}\theta_{ks}\ps_{kslm}.
\end{split}
\end{equation}
The equality \eqref{spin2} can be written in the form
\begin{equation}\label{pin1}
-2\Big(\delta T_{lm}+\frac76d^{\sb}\theta_{lm}\Big)+\Big(\delta T_{js}+\frac76d^{\sb}\theta_{js}\Big)\ps_{jslm}=\frac13\Big[d^{\sb}T_{kjsm}\ps_{kjsl}-d^{\sb}T_{kjsl}\ps_{kjsm}\Big].
\end{equation}
Using \eqref{p1}, we calculate
\begin{equation}\label{pin2}
\Big[d^{\sb}T_{kjsm}\ps_{kjsl}-d^{\sb}T_{kjsl}\ps_{kjsm}\Big]\ps_{mlab}
=-6\Big[ d^{\sb}T_{kjsa}\ps_{kjsb}-d^{\sb}T_{kjsb}\ps_{kjsa} \Big] .
\end{equation}
The skew-symmetric part of \eqref{spin1} together with \eqref{pin2} yield
\begin{equation}\label{pin3}
\begin{split}
-2\Big(\delta T_{lm}-\frac76d^{\sb}\theta_{lm}\Big)=-\frac1{12}\Big[d^{\sb}T_{kjsm}\ps_{kjsl}-d^{\sb}T_{kjsl}\ps_{kjsm}\Big];\\
\Big(\delta T_{ab}-\frac76d^{\sb}\theta_{ab}\Big)\ps_{ablm}=\frac1{24}\Big[d^{\sb}T_{kjsb}\ps_{kjsa}-d^{\sb}T_{kjsa}\ps_{kjsb}\Big]\ps_{ablm}\\
=-\frac3{12}\Big[d^{\sb}T_{kjsm}\ps_{kjsl}-d^{\sb}T_{kjsl}\ps_{kjsm}\Big].
\end{split}
\end{equation}
The sum of the two equalities in \eqref{pin3} implies
\begin{equation}\label{pin4}
\begin{split}
-2\Big(\delta T_{lm}-\frac76d^{\sb}\theta_{lm}\Big)+\Big(\delta T_{ab}-\frac76d^{\sb}\theta_{ab}\Big)\ps_{ablm}=-\frac1{3}\Big[d^{\sb}T_{kjsm}\ps_{kjsl}-d^{\sb}T_{kjsl}\ps_{kjsm}\Big].
\end{split}
\end{equation}
Summing up \eqref{pin1} and \eqref{pin4} to get
$
2\delta T_{lm}-\delta T_{ab}\ps_{ablm}=0.
$ 
Hence, $\delta T\in \Lambda^2_{21}\cong\frak{spin}(7)$. The second inclusion follows from \eqref{sdeltaT} and the just proved first one.

Suppose $d\theta=0$. Then \eqref{sdeltat1} follows from  \eqref{ttg2}, \eqref{in1}, \eqref{pin1} and \eqref{spin1} which proves a).

If $d^{\sb}\theta=0$ then the first two identities in \eqref{sdeltat12} are consequences of  \eqref{ttg2}, \eqref{in1} and \eqref{pin1}. Now \eqref{spin1}  implies $\delta T=\frac76d^{\sb}\theta=0$.
\end{proof}
\begin{prop}\label{spin71}
Let $(M,\ps)$ be  a $Spin(7)$ manifold, the  curvature of the torsion  connection $\sb$   is a $Spin(7)$-instanton and the four form $d^{\sb}T=0$. 

 The co-differential of the torsion is given by 
\begin{equation}\label{sdeltat}
\delta T=\frac73\sb\theta\in \Lambda^2_{21}\cong\frak{spin}(7)
\end{equation}
and the scalar curvature of the torsion connection is constant, $d(Scal)=0$.

In particular, the Lee vector field corresponding to the Lee form $\theta$ is Killing and $\delta\theta=0$.
\end{prop}
\begin{proof}
The condition $d^{\sb}T=0$  together with \eqref{spin1} implies  $\delta T=\frac73\sb\theta\in \Lambda^2_{21}\cong\frak{spin}(7)$ because of Proposition~\ref{spin7}. Consequently the Lee vector field  $\theta$ is a Killing vector field. 

Multiply \eqref{ins2} with $\ps_{ijab}$, use the instanton condition, \eqref{srri} and \eqref{ins1} to get
 \begin{multline}\label{sins3}
 2\sb_iT_{jkl}\ps_{ijab}=\Big[-\sb_kT_{lij}+\sb_lT_{kij}\Big]\ps_{ijab}\\=2R_{abkl}-2R_{klab}=-2\Big[\sb_kT_{lab}-\sb_lT_{kab}\Big]=2\Big[\sb_aT_{bkl}-\sb_bT_{akl}\Big].
 \end{multline}
 We will use \eqref{e1}. The last term vanishes because $dT_{iabc}T_{abc}=2\sigma^T_{iabc}T_{abc}=0$.
 
 For the fourth term we have applying \eqref{iii} and \eqref{sdeltat} that $\delta T_{ab}T_{abl}=2\sb_a\delta T_{al}=\frac{14}3\sb_a\sb_a\theta_l$.
 
For the  first term we have from \eqref{sricg2}  $d(Scal)_l=\frac{49}{18}\sb_l||\theta||^2-\frac13\sb_l||T||^2$ since $\delta\theta=0$.
 
Finally, for the second term we calculate from \eqref{spbi2} using \eqref{spin1}, \eqref{sins3}, \eqref{inst8} and \eqref{iii}
 \begin{multline*}
 -2\sb_mRic_{lm}=\frac13\sb_m\sigma^T_{ijkl}\ps_{ijkm}+\frac73\sb_m\sb_l\theta_m=T_{skl}\sb_mT_{ijs}\ps_{ijkm}+T_{ijs}\sb_mT_{skl}\ps_{ijkm}-\frac73\sb_m\sb_m\theta_l\\
 =\frac73\sb_s\theta_kT_{skl}-\frac73\sb_m\sb_m\theta_l+2T_{ijs}\sb_iT_{jsl}=\frac73\sb_m\sb_m\theta_l+\frac13\sb_l||T||^2,
 \end{multline*}
 where we apply \eqref{sdeltat} and \eqref{iii} to achieve the last equality.
 
 Hence, \eqref{e1} takes the following form
 \begin{equation}\label{es1}
 \begin{split}
0= \frac{49}{18}\sb_l||\theta||^2-\frac13\sb_l||T||^2+\frac73\sb_m\sb_m\theta_l+\frac13\sb_l||T||^2+\frac16\sb_l||T||^2+\frac{14}3\sb_a\sb_a\theta_l\\=\frac{49}{18}\sb_l||\theta||^2+7\sb_a\sb_a\theta_l+\frac16\sb_l||T||^2.
\end{split}
\end{equation}
 The Ricci identity for $\sb$ together with the Killing condition for the Lee vector field, \eqref{sdeltat}, \eqref{iii} and \eqref{spbi1} imply
 \begin{equation}\label{skil}
 \begin{split}
 -\sb_a\sb_a\theta_l=\sb_a\sb_l\theta_a=-R_{alas}\theta_s-T_{als}\sb_s\theta_a=Ric_{ls}\theta_s-2\sb_a\sb_a\theta_l\\=
 -\frac16\sigma^T_{abcl}\theta_l\ps_{abcs}-\frac7{12}\sb_l||\theta||^2-2\sb_a\sb_a\theta_l.
 \end{split}
 \end{equation}
 Using $\theta\lrcorner T\in\Lambda^2_{21}\cong\frak{spin}(7)$ and \eqref{storcy2}, we calculate
 \begin{equation}\label{vann}
 \begin{split}\frac13\sigma^T_{abcl}\theta_l\ps_{abcs}=T_{abd}\ps_{abcs}T_{dcl}\theta_l=T_{dcs}T_{dcl}\theta_l-\frac12T_{abs}\ps_{abdc}T_{dcl}\theta_l-\frac76\theta_p\ps_{pdcs}T_{dcl}\theta_l\\=T_{dcs}T_{dcl}\theta_l-T_{abs}T_{abl}\theta_l-\frac76\theta_pT_{psl}\theta_l=0.
 \end{split}
 \end{equation}
 The  identity \eqref{vann} and  \eqref{skil} imply
$
 \sb_a\sb_a\theta_l=-\frac7{12}\sb_l||\theta||^2
$ 
which combined with  \eqref{es1}  yields
 \begin{equation}\label{ssss}
 0=\sb_l\Big[-\frac{49}{36}||\theta||^2+\frac16||T||^2\Big]=-\frac12d(Scal)_l.
 \end{equation}
This completes the proof of the proposition.
\end{proof}
\subsection{Proof of Theorem~\ref{main2}}
\begin{proof}
Suppose $\sb T=0$. Then $d^{\sb}T=\delta T=0$ and \eqref{dh} implies \eqref{inst6}. Therefore  $\sb dT=2\sb\sigma^T=0$ the Ricci tensor of the torsion connection is symmetric, because of \eqref{rics}, and $\sb$-parallel with constant scalar curvature. Moreover, \eqref{4form} shows that the characteristic curvature $R\in S^2\Lambda^2$ and therefore $R$ is a $Spin(7)$-instanton since $\sb\ps=0$.

For the converse, we start with the following
\begin{lemma}
If on a $Spin(7)$ manifold the curvature of the torsion connection is a $Spin(7)$-instanton then the following equality holds true
\begin{equation}\label{sbi24}
\sb_iR_{iplm}=\frac76\theta_rR_{rplm}.
\end{equation}
\end{lemma}

\begin{proof}
Multiplying the second Bianchi identity \eqref{bi22} with $\ps_{ijkp}$ and using the $Spin(7)$-instanton conditions \eqref{srri}, we obtain 
\begin{equation}\label{sbi23}
6\sb_iR_{iplm}+3T_{ijs}R_{sklm}\ps_{ijkp}=0.
\end{equation}
An application of  \eqref{storcy2} together with \eqref{srri} to the second term in \eqref{sbi23} yields
\begin{multline*}
T_{ijs}\ps_{ijkp}R_{sklm}=\Big[-T_{ijk}\ps_{ijps}-T_{ijp}\ps_{ijsk}+2T_{skp}-\frac73\theta_r\ps_{rskp} \Big]R_{sklm}\\
=-T_{ijk}\ps_{ijps}R_{sklm}-2T_{ijp}R_{ijlm}+2T_{skp}R_{sklm}-\frac{14}3\theta_rR_{rplm}=-T_{ijs}\ps_{ijkp}R_{sklm}-\frac{14}3\theta_rR_{rplm}.
\end{multline*}
The last identity can be written as
\begin{equation}\label{sm2}
T_{ijs}\ps_{ijkp}R_{sklm}=-\frac73\theta_rR_{rplm}.
\end{equation}
Substitute \eqref{sm2} into \eqref{sbi23} to get \eqref{sbi24} which proves the lemma.
\end{proof}
Further, we multiply \eqref{sbi24} with $T_{plm}$, using \eqref{1bi1}  the conditions $dT=2\sigma^T$ and $d^{\sb}T=0$
\begin{multline}\label{sbi25}
0=3\Big[\sb_iR_{iplm} - \frac76\theta_iR_{iplm}\Big]T_{plm}=\sb_i\Big[-\sigma^T_{plmi}+\sb_iT_{plm}\Big]T_{plm}
 -\frac76\theta_i\Big[-\sigma^T_{plmi}+\sb_iT_{plm}\Big]T_{plm}\\
=-\sb_i\sigma^T_{plmi}T_{plm}+T_{plm}\sb_i\sb_iT_{plm} - \frac7{12}\sb_{\theta}||T||^2=\sigma^T_{plmi}\sb_iT_{plm}+T_{plm}\sb_i\sb_iT_{plm} - \frac7{12}\sb_{\theta}||T||^2\\=
\frac14\sigma^T_{plmi}d^{\sb}T_{iplm}+T_{plm}\sb_i\sb_iT_{plm} - \frac7{12}\sb_{\theta}||T||^2=T_{plm}\sb_i\sb_iT_{plm} - \frac7{12}\sb_{\theta}||T||^2.
\end{multline}
A substitution of \eqref{lapt} into \eqref{sbi25} yields
\begin{equation}\label{slap}
\Delta||T||^2 +\frac76\sb_{\theta}||T||^2=-2||\sb T||^2\le 0.
\end{equation}
Since $M$ is compact  we may apply  the strong maximum principle to \eqref{slap} (see e.g. \cite{YB,GFS}) to achieve  $\sb T=0$ which completes the proof of Theorem~\ref{main2}.
\end{proof}
As a consequence of the proof of Theorem~\ref{main2}, we obtain from \eqref{slap}
\begin{cor}\label{smain3}
Let $(M,\ps)$ be  a $Spin(7)$ manifold. The next two conditions are equivalent:
\begin{itemize}
\item[a)] The torsion 3-form is parallel with respect to the torsion  connection, $\sb T=0$.
\item[b)] The  curvature of the torsion  connection $\sb$   is a $Spin(7)$-instanton, the norm of the torsion is constant, $d||T||^2=0$ and \eqref{inst6} holds true.
\end{itemize}
\end{cor}
\subsection{Proof of Theorem~\ref{mainsp0}}
\begin{proof}
We observe that  the  condition $d\theta=0$ together with \eqref{sdeltat1} and \eqref{sdeltat} imply 
\[ \delta T_{ij}\theta_j=\frac73\theta_j\sb_i\theta_j=\frac76\sb_i||\theta||^2=\frac76\theta_sT_{sij}\theta_j=0\]
which shows that the norm of $\theta$ is a constant.  Using \eqref{ssss},  we conclude that the norm of the torsion is constant and Corollary~\ref{smain3} completes the proof of Theorem~\ref{mainsp0}.
\end{proof}
The next example shows  a $Spin(7)$ manifold with $\sb$-parallel torsion and non-closed Lee form.
\begin{exam}\label{exdtit}
We take the next example of a $G_2$ manifold with parallel torsion with respect to the characteristic connection and non-closed Lee form from \cite[Example~7.7]{IS1}.

The group $G=SU(2)\times SU(2)\times S^1$ has a Lie algebra $g=\frak{su}(2)\oplus \frak{su}(2)\oplus\mathbb R$ and structure equations   
\[de_1=e_{23},\quad de_2=e_{31},\quad de_3=e_{12},\quad de_4=e_{56},\quad de_5=e_{64}\quad de_6=e_{45}, \quad de_7=0.\]
The  left-invariant $G_2$ structure $\p$ defined by \eqref{g2} generates the bi-invariant metric and the characteristic  connection is the flat left invariant  Cartan  connection with closed  nad $\sb$ parallel torsion $T=-[.,.]$.  According to \cite[Example~7.7]{IS1} the $G_2$ structure  is strictly integrable with $\sb$-parallel closed torsion 3-form $T=e_{123}+e_{456}$ and non closed Lee form $\theta=e_4-e_3, d\p\wedge\p=0$. 

Consider the group $S^1\times G=S^1\times SU(2)\times SU(2)\times S^1$ with the $Spin(7)$ structure defined by \eqref{s1},
\[-\Omega=\ps=-e_0\wedge\p-*\p,\]
where $e_0$ is the closed 1-form on the first factor $S^1$.

 According to \cite[Theorem~5.1]{II} the torsion $T^8$ of $\Omega$ is equal to the characteristic torsion $T^7$ of $\p$ and is parallel with respect to the torsion connection of $\Omega$ which is the bi-invariant flat Cartan connection on the group manifold $S^1\times G=S^1\times SU(2)\times SU(2)\times S^1$. Moreover,  the Lee form $\theta^8$ of the $Spin(7)$ structure $\Omega$ is connected with the Lee form $\theta^7$ of the $G_2$ structure $\p$ by
 \[ \theta^8=\frac76\theta^7+\frac17(d\p,*\p)e_0=\frac76(e_4-e_3), \qquad d\theta^8\not=0.\]
\end{exam}

\subsection{Proof of Theorem~\ref{mainsp}}
\begin{proof}
Clearly, if $\sb T=0$ then $0=d^{\sb}T=\delta^{\sb}T=\delta T$, where we used \eqref{inst6}. Moreover,  the torsion connection is a $Spin(7)$-instanton  because $R\in S^2\Lambda^2$ due to \eqref{4form} and $Hol(\sb)\in \frak{spin}(7)\cong\Lambda^2_{21}$. 


To complete the proof of Theorem~\ref{mainsp} we observe, that
under the conditions of the theorem,   Proposition~\ref{spin71} implies $\sb\theta=0$. In particular the norm of $\theta$ is constant. Using \eqref{ssss},  we conclude that the norm of the torsion is constant and Corollary~\ref{smain3} completes the proof of Theorem~\ref{mainsp}.
\end{proof}
On a balanced $Spin(7)$ manifold the Lee form vanishes 
and   Corollary~\ref{mainspc} follows from Theorem~\ref{mainsp}.

 \subsection{Compact  Gauduchon $Spin(7)$ manifolds. Proof of Theorem~\ref{spga}} 
In this subsection, we recall the notion of conformal deformations of a given $Spin(7)$ structure $\ps$  from \cite{Fer,I1,Kar1} and prove Theorem~\ref{spga}. 

Let $\bar\ps=e^{4f}\ps$ be a conformal deformation of $\ps$. The induced metric $\bar g=e^{2f}g$. 
The Lee forms are connected by $\bar\theta=\theta+4df$. Consequently, if the Lee form is closed then it remains closed for all conformally related $Spin(7)$-structures. Using the expression of the Gauduchon theorem 
 in terms of a Weyl structure \cite [Appendix 1]{tod},
one can find, in a unique way, a conformal $Spin(7)$ structure such that the corresponding Lee 
1-form is coclosed with respect to the induced metric due to \cite[Theorem~4.3]{I1}.
\begin{proof} 
Now we prove Theorem~\ref{spga} following the proof of Theorem~\ref{RG1}. The conditions of the theorem together with Proposition~\ref{spin7} a) imply \eqref{sdeltat1} holds true. 
Applying \eqref{sdeltat1} we write \eqref{iii}  in the form
\begin{equation}\label{ntss1p}
\sb_i\sb_i\theta_j-\sb_i\sb_j\theta_i=\frac12d^{\sb}\theta_{ab}T_{abj}.
\end{equation}
The Ricci identity  \eqref{rrii}
substituted into \eqref{ntss1p} yields 
\begin{equation}\label{gafinsp}
\sb_i\sb_i\theta_j+\sb_j\delta\theta-Ric_{js}\theta_s=0. 
\end{equation}

We proceed as in the proof of Theorem~\ref{RG1}. 
Multiplying  the both sides of \eqref{gafinsp} with $\theta_j$, use  $\delta\theta=0$ together with the identity \eqref{llii} we  derive  \eqref{gafinf} holds true also in this case.

Further, we calculate from \eqref{sricg2} with the help of \eqref{dh}  that
\begin{equation}\label{srri4}
\begin{split}
Ric_{ij}\theta_i\theta_j=-\frac1{12}dT_{abci}\ps_{abcj}\theta_i\theta_j-\frac76\theta_i\theta_j\sb_i\theta_j\\=
-\frac1{12}\Big[2\sigma^T_{abci}\ps_{abcj}+3\sb_aT_{bci}\ps_{abcj}+21\sb_i\theta_j  \Big]\theta_i\theta_j=-\frac1{12}\Big[3\sb_aT_{bci}\ps_{abcj}+21\sb_i\theta_j  \Big]\theta_i\theta_j,
\end{split}
\end{equation}
where we applied \eqref{vann}, we do this  since $\theta\lrcorner T\in\Lambda^2_{21}\cong\frak{spin}(7)$ by \eqref{sdeltat1}.

We obtain from \eqref{spin1} and  \eqref{sdeltat1} that
\begin{equation*}
\sb_aT_{bci}\ps_{abcj}=\frac73\sb_i\theta_j+\frac{14}3\sb_j\theta_i
\end{equation*}
which substituted into \eqref{srri4} gives
\begin{equation}\label{spinn}
Ric_{ij}\theta_i\theta_j=-\frac72\sb_i\theta_j\theta_i\theta_j=-\frac74\theta_i\sb_i||\theta||^2.
\end{equation}

Insert \eqref{spinn} into \eqref{gafinf} to obtain 
\begin{equation}\label{sgafinf1}
\Delta||\theta||^2-\frac74\theta_i\sb_i||\theta||^2=-2||\sb\theta||^2\le 0.
\end{equation}
We  apply  the strong maximum principle to \eqref{sgafinf1} (see e.g. \cite{YB,GFS}) to achieve $d||\theta||^2=\sb\theta=0$. 

Consequently, \eqref{sdeltat1} implies $\delta T=0$ and \eqref{spin1} leads to $d^{\sb} T_{ijkl}\ps_{ijkm}=0$ and therefore $d^{\sb} T\in\Lambda^4_{27}$ is self-dual.
\end{proof}

\section{Closed torsion, Hull Spin(7) instanton}
We recall that the $Spin(7)$-Hull connection $\sb^h$ is defined to be the metric connection with torsion $-T$, where $T$ is the torsion of the Spin(7) torsion  connection,
\begin{equation}\label{hu}\sb^h=\LC-\frac12T=\sb-T.
\end{equation}
Concerning the $Spin(7)$-Hull connection, we prove the following
\begin{thrm}\label{hul}
Let $(M,\ps$) be a compact $Spin(7)$ manifold.
The curvature $R^h$ of the $Spin(7)$-Hull connection $\sb^h$ is a $Spin(7)$-instanton if and only if the torsion is closed, $dT=0$.
\end{thrm}
\begin{proof}
We start with the general well-known formula for the curvatures of two metric connections with totally skew-symmetric torsion $T$ and $-T$, respectively, see e.g. \cite{MS}, which applied to the curvatures of the characteristic connection and the $Spin(7)$-Hull connection reads
\begin{equation}\label{hust}
R(X,Y,Z,V)-R^h(Z,V,X,Y)=\frac12dT(X,Y,Z,V).
\end{equation}
If $dT=0$ the result was   observed  in \cite{MS}. Indeed, in this case  the $Spin(7)$-Hull connection is a $Spin(7)$ instanton since $\sb\ps=0$ and the holonomy group of $\sb$ is contained in the Lie algebra $\frak{spin}(7)$ \cite{MS}.

For the converse, \eqref{hust} yields
\begin{equation}\label{huin1}
\begin{split}
dT_{iabc}\ps_{jabc}=R_{iabc}\ps_{jabc}+R^h_{bcai}\ps_{jabc}=2R_{iaja}+2R^h_{jaai}=-2Ric_{ij}+2Ric^h_{ji}=0,
\end{split}
\end{equation}
where $Ric^h$ is the Ricci tensor of the $Spin(7)$-Hull connection and  the trace of \eqref{hust} gives $Ric(X,V)-Ric^h(V,X)=0.$
The identity \eqref{huin1} shows that the 4-form $dT\in\Lambda^4_{27}$ by Proposition~\ref{427} and, in particular, it is self-dual, $*dT=dT$. Therefore, we have 
\[\delta dT=-*d**dT=-*d^2T=0. \]
Multiply with $T$ and integrate over the compact space we obtain 
\[0=\frac1{24}\int_Mg(\delta dT,T)vol = \frac1{24}\int_M||dT||^2vol.\]
Hence, $dT=0$.
\end{proof}

If the torsion is closed then the $Spin(7)$ manifold is said to be \emph{strong}.
\begin{thrm}\label{sp7clost}
Let $(M,\Psi)$ be a  compact  $Spin(7)$ manifold with  closed torsion. 

The next conditions are equivalent:
\begin{itemize}
\item[a)] The torsion is $\sb$-parallel.
\item[b)] The curvature of $\sb$ is  a $Spin(7)$-instanton.
\end{itemize}
\end{thrm}
\begin{proof}
If $\sb T=0$ then $R\in S^2\Lambda^2$ by \cite[Lemma~3.4]{I} which shows that $R$ is a $Spin(7)$-instanton.

For the converse, let $R$ be a $Spin(7)$-instanton. Then \eqref{sbi24} holds true.

We multiply \eqref{sbi24} with $T_{plm}$, using \eqref{1bi1}  and $dT=0$ to calculate 
\begin{multline}\label{bi25cl}
0=3\Big[\sb_iR_{iplm}-\frac76\theta_iR_{iplm}\Big]T_{plm}=T_{plm}\sb_i\sb_iT_{plm}
-\frac76\theta_i\sb_iT_{plm}T_{plm}\\=T_{plm}\sb_i\sb_iT_{plm}-\frac7{12}\sb_{\theta}||T||^2.
\end{multline}
A substitution of \eqref{lapt} into \eqref{bi25cl} yields that \eqref{slap} holds true.
Since $M$ is compact  we may apply  the strong maximum principle to \eqref{lap} (see e.g. \cite{YB,GFS}) to achieve  $\sb T=0$ which completes the proof of the theorem.
\end{proof}

\vskip 0.5cm
\noindent{\bf Acknowledgements} \vskip 0.1cm
\noindent 
\noindent 
The research of S.I. and A.P.  is partially  
supported   by Contract KP-06-H72-1/05.12.2023 with the National Science Fund of Bulgaria and  Contract 80-10-61 / 27.05.2025   with the Sofia University "St.Kl.Ohridski". 
L.U. is partially supported by grant PID2023-148446NB-I00, funded by MICIU/AEI/10.13039/501100011033, and by grant E22-23R ``Algebra y Geometr\'ia'' (Gobierno de Arag\'on/FEDER).

\vskip 0.5cm



\begin{thebibliography}{99}
\bibitem{AFer} I. Agricola, A. Ferreira, \emph{Einstein manifolds with skew torsion},
The Quarterly Journal of Mathematics, \textbf{65}, Issue 3,  (2014),  717-741, https://doi.org/10.1093/qmath/hat050

\bibitem{AFF}  I. Agricola, A. Ferreira, Th. Friedrich, \emph{The classification of naturally reductive homogeneous spaces in dimensions $n\le 6$}, Diff. Geom. Appl., \textbf{39} (2015), 59-92.


\bibitem{AMP}  A. Ashmore, R. Minasian, Y. Proto, \emph{Geometric flows and supersymmetry}, 
Commun. Math. Phys. \textbf{405} (2024), 16. https://doi.org/10.1007/s00220-023-04910-7.

\bibitem{BO}
G. Ball, G. Oliveira, 
\emph{Gauge theory on Aloff–Wallach spaces}, 
Geom. Topol. 23 (2019), 685--743.




\bibitem{Bo} E. Bonan,  \emph{Sur le vari\'et\'es riemanniennes a groupe d'holonomie $G_2$ ou $Spin(7)$}, C. R. Acad. Sci. Paris
262 (1966), 127-129.


\bibitem{Br} R.L. Bryant, \emph{Metrics with exceptional holonomy},  Ann. Math. 126 (1987), 525-576.

\bibitem{Br1} R.L. Bryant, \emph{Some remarks on $G_2$ structures}, Proceedings of G\"okova Geometry-Topology
Conference 2005, 75-109.

\bibitem{BS} R. Bryant, S.Salamon, {\em On the construction of some
  complete metrics with exceptional holonomy}, Duke Math. J. {\bf 58}
  (1989), 829-850.

\bibitem{Cabr} F.M. Cabrera, \emph{On Riemannian Manifolds with $G_2$-Structure}, Bolletino U.M.I. (7) 10-A
(1996), 99-112.

\bibitem {C1} F. Cabrera, {\em On Riemannian manifolds with $Spin(7)$-structure}, Publ. Math. Debrecen {\bf 46} (3-4) (1995), 271-283.


\bibitem {CSal} S. Chiossi, S. Salamon, {\em The intrinsic torsion of SU(3) and
  $G_2$-structures}, Differential Geometry, Valencia 2001, World Sci.
  Publishing, 2002, pp. 115-133.
  
\bibitem{Clarke}  
A. Clarke, 
\emph{Instantons on the exceptional holonomy manifolds of Bryant and Salamon}, 
J. Geom. Phys. 82 (2014), 84--97.

  
\bibitem{CGFT}  A. Clarke, M. Garcia-Fernandez, C. Tipler, \emph{ T-Dual solutions and infinitesimal moduli of the $G_2$-Strominger system},   Adv. Theor. Math. Phys. 26 (2022), no. 6, 1669-1704. 
   
   
\bibitem{CdBM} 
 A. Clarke, V. del Barco, A.J. Moreno, 
 \emph{$G_2$-instantons on 2-step nilpotent Lie groups},  Adv. Theor. Math. Phys.  29 (2025), no. 3, 751--813.

 

\bibitem{CMS} R. Cleyton, A. Moroianu, U. Semmelmann, \emph{Metric connections with parallel skew-symmetric torsion}, Advances in Mathematics 378 (2021), 107519, https://doi.org/10.1016/j.aim.2020.107519




   



 \bibitem {Fer} 
 M. Fern\'andez, {\em A classification of Riemannian manifolds with structure group
$Spin(7)$}, Ann. Mat. Pura Appl. {\bf 143} (1982), 101-122.

\bibitem{FG} M. Fern\'andez, A. Gray, \emph{Riemannian manifolds with structure group $G_2$}, Ann. Mat. Pura Appl. (4) {\bf 132} (1982), 19-45. 

\bibitem{FIUVdim7-8} 
M. Fern\'andez, S. Ivanov, L. Ugarte, R. Villacampa, 
\emph{Compact supersymmetric solutions of the heterotic equations of motion in dimensions 7 and 8}, 
Adv. Theor. Math. Phys. 15 (2011), no. 2, 245--284.


\bibitem{FUg} M. Fern\'andez, L. Ugarte, \emph{Dolbeault Cohomology for $G_2$-Manifolds}, Geom. Dedicata 70 (1998), 57--86.


\bibitem{F} Th. Friedrich, \emph{$G_2$-manifolds with parallel characteristic torsion}, Diff. Geom. Appl. 25 (6) (2007), 632-648.

\bibitem{Fr} 
Th. Friedrich, \emph{On types of non-integrable geometries}, Rend. Circ. Mat. Palermo \textbf{71} (2003), 99-113.



\bibitem{FI} Th. Friedrich, S. Ivanov,  \emph{Parallel spinors and connections with skew-symmetric
torsion in string theory},  Asian Journ.
Math. {\bf 6} (2002), 303-336.

\bibitem{FI1} Th. Friedrich, S. Ivanov, {\em Killing spinor equations in
  dimension 7 and geometry of integrable $G_2$ manifolds}, J. Geom. Phys. 
  {\bf 48} (2003), 1-11.
  


\bibitem{GFS} M. Garcia-Fernandez and J. Streets, Generalized Ricci Flow. AMS University Lecture Series 76, 2021. 


  
  \bibitem {GMW} J. Gauntlett, D. Martelli, D. Waldram, {\em Superstrings with
  Intrinsic torsion}, Phys. Rev. {\bf D69} (2004) 086002.

\bibitem {GKMW} J. Gauntlett, N. Kim, D. Martelli, D. Waldram, {\em Fivebranes
  wrapped on SLAG three-cycles and related geometry}, JHEP 0111 (2001) 018.

\bibitem {GMPW} J.P. Gauntlett, D. Martelli, S. Pakis, D. Waldram,
{\em  G-Structures and Wrapped NS5-Branes}, Commun. Math. Phys.
{\bf 247} (2004), 421-445.

\bibitem{Gibb} G.W. Gibbons, D.N. Page, C.N. Pope, {\em Einstein metrics
on $S^3,\mathbb R^3,$ and $\mathbb R^4$ bundles}, Commun. Math. Phys. {\bf 127}
(1990), 529-553.

\bibitem{Gr} A. Gray, \emph{Vector Cross Products on Manifolds}, Trans. Amer. Math. Soc. 141 (1969), 463-504;
correction 148 (1970), 625.




  
 \bibitem{Hull} C.M. Hull, \emph{Compactification of the heterotic superstrings}, Phys. Lett. B 178 (1986) 357-364.



\bibitem{II} P. Ivanov and S. Ivanov, \emph{SU(3) instantons and G2, Spin(7) heterotic string solitons},
Commun. Math. Phys. 259 (2005) 79-102.

\bibitem{I} S.~Ivanov, \emph{Geometry of Quaternionic K\"ahler connections with torsion}, J. Geom.  Phys. {\bf 41} (2002), 235-257.

\bibitem {I1} S.~Ivanov, {\em Connection with torsion, parallel spinors and geometry of Spin(7) manifolds}, Math. Res. Lett.  {\bf 11} (2004), no. 2-3, 171--186.

\bibitem{Iv} S.~Ivanov, \emph{Heterotic supersymmetry, anomaly cancellation and equations of motion}, Phys. Lett. B,
685(2-3):190-196, 2010.



\bibitem{IP2}
S.~Ivanov and G.~Papadopoulos,
  \emph{Vanishing theorems and string backgrounds},
  Class. Quant. Grav. {\bf 18} (2001) 1089.


\bibitem{IPP} S. Ivanov, M. Parton, P. Piccinni, \emph{Locally conformal parallel $G_2$ and $Spin(7)$ manifolds},  Math. Res. Lett. {\bf 13} (2006), 167-177.

\bibitem{Iap} S. Ivanov, A. Petkov, \emph{The Riemannian curvature identities for the torsion connection on Spin(7)-manifold and generalized Ricci solitons},   Math. Nachrichten, 298(9):2906-2925, 2025, 
https://doi.org/10.1002/mana.12021, arXiv:2307.06438., 

\bibitem{IS} S. Ivanov, N. Stanchev, \emph{The Riemannian Bianchi identities  of metric connections  with  skew  torsion and generalized Ricci solitons},  Result. Math. \textbf{ 79}, No. 8, Paper No. 270, 20 p. (2024). https://doi.org/10.1007/s00025-024-02302-4.

\bibitem{IS1} S. Ivanov, N. Stanchev, \emph{The Riemannian curvature identities of a $G_2$  connection with skew-symmetric torsion and generalized Ricci solitons},  Quarterly Journal of Mathematics, 76(2):757-783,
2025, DOI: 10.1093/qmath/haaf031, arXiv:2307.05619.


\bibitem{J1} D. Joyce, {\em Compact Riemannian 7-manifolds with holonomy
  $G_2$. I }, J.Diff. Geom.  43 (1996), 291-328.

\bibitem{J2} D. Joyce, {\em Compact Riemannian 7-manifolds with holonomy
  $G_2$. II}, J.Diff. Geom.  43 (1996), 329-375.

\bibitem{J3} D. Joyce, Compact Riemannian manifolds with special holonomy,
  Oxford University Press, 2000.

\bibitem{Kar1}   S. Karigiannis, \emph{ Deformations of $G_2$ and $Spin(7)$  Structures on Manifolds}, Canadian Journal of Mathematics 57 (2005), 1012-1055. 


\bibitem{Kar}  S. Karigiannis, \emph{ Flows of $G_2$ Structures, I},  Quarterly Journal of Mathematics 60 (2009), 487-522.

\bibitem{Kar2} S. Karigiannis, \emph{Flows of Spin(7)-structures},  Proceedings of the 10th Conference on Differential Geometry and its Applications: DGA 2007; World Scientific Publishing, (2008), 263-277.

\bibitem{KenStr} 
A. Kennon, J. Streets, \emph{The canonical symmetry reduction of string backgrounds}, arXiv:2511.20773.


\bibitem {Kov} A. Kovalev, {\em Twisted connected sums and special
 Riemannian holonomy}, J. Reine Angew. math. {\bf 565} (2003), 125-160.

\bibitem {LM} B. Lawson, M.-L.Michelsohn, Spin Geometry, Princeton   University Press, 1989.


\bibitem{LO}
J.D. Lotay, G. Oliveira, 
\emph{SU$(2)^2$-invariant $G_2$-instantons}, 
Math. Ann. 371 (2018), 961--1011.


\bibitem{MS} D. Martelli and J. Sparks, \emph{Non-K\"ahler heterotic rotations}, Adv. Theor. Math. Phys. 15
(2011), 131-174.

\bibitem{Mer}  L. Martin-Merchan, \emph{Spinorial classification
of Spin(7) structures}, Ann. Sc. Norm. Super.
Pisa Cl. Sci. (5) Vol. XXI (2020), pp. 873-910.


\bibitem{MSem} 
A. Moroianu, U. Semmelmann, 
\emph{$G_2$-structures with parallel skew-symmetric torsion}, arXiv:2510.02899. 


\bibitem{MSh} V. Mu\~noz, C.S. Shahbazi, \emph{Transversality of the
moduli space of Spin(7)-instantons},  Rev. Math. Phys. 32 (2020), no. 5, 2050013, 47 pp.

\bibitem{OLS}   X. de la Ossa, M. Larfors, E.E. Svanes,
\emph{Exploring SU(3) Structure Moduli Spaces with Integrable G2
Structures}, Adv. Theor. Math. Phys. 19 (2015), no. 4, 837-903.

\bibitem{XS} X. de la Ossa, E.E. Svanes, \emph{Holomorphic Bundles
and the Moduli Space of N=1 Heterotic Compactifications},  J. High Energy Phys. 2014, no. 10, 123, front matter+54 pp. 

\bibitem{Oss1} X. de la Ossa, M. Larfors and E.E. Svanes, \emph{The Infinitesimal Moduli Space of Heterotic $G_2$ Systems}, Commun. Math. Phys. 360 (2018) 727.

\bibitem{Oss2} X. de la Ossa, M. Larfors, M. Magill and E.E. Svanes, \emph{Superpotential of three dimensional $N=1$ heterotic supergravity}, JHEP 01 (2020) 195.



\bibitem{Pap}
G. Papadopoulos, 
On the rigidity of special and exceptional geometries with torsion a closed 3-form, arXiv:2511.20568. 









\bibitem{SW}
H.N. S\'a Earp, T. Walpuski, 
\emph{$G_2$-instantons over twisted connected sums},
Geom. Topol. 19 (2015), 1263--1285.


\bibitem{Sal} S. Salamon, Riemannian geometry and holonomy groups, Pitman   Res.  Notes Math. Ser., 201 (1989).


\bibitem{USem} 
U. Semmelmann, \emph{Conformal killing forms on Riemannian manifolds}, Math. Z. \textbf{243} (2003), 503-527.


  \bibitem{Str}
A.~Strominger, \emph{Superstrings With Torsion},
Nucl.\ Phys.\ B {\bf 274} (1986), 253.

\bibitem{tod} K.P. Tod, \emph{Compact 3-dimensional Einstein-Weyl structures}, J. Lond. Math. Soc. \textbf{45} (1992), 341-351.

 \bibitem {Ug} L. Ugarte, {\em Coeffective Numbers of Riemannian 8-manifolds with Holonomy in} $Spin(7)$, Ann. Glob. Anal. Geom. {\bf 19} (2001), 35-53.
 
  \bibitem{Waldron}
 A. Waldron, 
 \emph{$G_2$-instantons on the $7$-sphere},  
 J. London Math. Soc. \textbf{106} (2022), 3711--3745.
 
 
 
 \bibitem{W}
 T. Walpuski, 
 \emph{$G_2$-instantons on generalised Kummer constructions}, 
 Geom. Topol. 17 (2013), 2345--2388.

\bibitem{Yan} 
K. Yano, 
\emph{Some remarks on tensor fields and curvature}, Ann. Math. (2) 55 (1952), 328-347.



\bibitem{YB} K. Yano and S. Bochner, Curvature and Betti Numbers, Ann. of Math. Studies 32, Princeton University Press, 1953.




\end{thebibliography}
\end{document}